\documentclass[11pt,a4paper]{article}
\usepackage[utf8]{inputenc}
\usepackage[usenames,dvipsnames]{xcolor}
\usepackage{amsmath}
\usepackage{amsthm}
\usepackage{amsfonts}
\usepackage{amssymb}
\usepackage{array}  
\usepackage{graphicx} 
\usepackage{color}
\usepackage[english]{babel}
\usepackage{mathrsfs}
\usepackage{graphicx}
\usepackage{dsfont}
\usepackage{geometry}
\usepackage[final]{hyperref} 

\theoremstyle{plain}
\newtheorem{thm}{Theorem}[section]

\newtheorem{theorem}[thm]{Theorem}
\newtheorem{corollary}[thm]{Corollary}
\newtheorem{proposition}[thm]{Proposition}

\newtheorem{lemma}[thm]{Lemma}

\theoremstyle{definition}
\newtheorem{definition}[thm]{Definition}

\newtheorem{condition}[thm]{Condition}
\newtheorem{remark}[thm]{Remark}

\theoremstyle{remark}

\newcommand{\N}{\mathbb{N}}
\newcommand{\R}{\mathbb{R}}

\newcommand{\Z}{\mathbb{Z}}

\newcommand{\prob}{\mathbb{P}}
\newcommand{\E}{\mathbb{E}}
\newcommand{\T}{\mathbb{T}}

\def\T{\mathbb{T}}
\newcommand{\un}{\mathds{1}}

\renewcommand{\textbf}[1]{\begingroup\bfseries\mathversion{bold}#1\endgroup}

\def\calA{\mathcal{A}}

\def\calC{\mathcal{C}}
\def\calD{\mathcal{D}}
\def\calE{\mathcal{E}}
\def\calF{\mathcal{F}}
\def\calG{\mathcal{G}}

\def\calL{\mathcal{L}}
\def\calM{\mathcal{M}}

\def\calR{\mathcal{R}}

\def\calT{\mathcal{T}}

\def\calX{\mathcal{X}}

\def\var{\mathop{\mathrm{Var}}}

\def\P{\mathbb{P}} 
\def\E{\mathbb{E}} 

\def \eps {\varepsilon}

\def\st{\ :\ }

\def\br#1{\left(#1\right)}
\def\brb#1{\left[#1\right]}

\numberwithin{equation}{section}
\def\crit{\textup{Crit}}

\def\thr{\textup{T}_\mathcal{A}}
\def\sad{\textup{S}_\mathcal{A}}

\title{Talagrand's inequality in planar Gaussian field percolation}
\author{Alejandro Rivera\thanks{Univ. Grenoble Alpes, UMR5582, Institut Fourier, 38610 Gières, France, supported by the ERC grant Liko No 676999}}
\date{\today}

\begin{document}
\maketitle

\begin{abstract}
Let $f$ be a stationary isotropic non-degenerate Gaussian field on $\R^2$. Assume that $f=q*W$ where $q\in L^2(\R^2)\cap C^2(\R^2)$ and $W$ is the $L^2$ white noise on $\R^2$. We extend a result by Stephen Muirhead and Hugo Vanneuville by showing that, assuming that $q*q$ is pointwise non-negative and has fast enough decay, the set $\{f\geq -\ell\}$ percolates with probability one when $\ell>0$ and with probability zero if $\ell\leq 0$. We also prove exponential decay of crossing probabilities and uniqueness of the unbounded cluster. To this end, we study a Gaussian field $g$ defined on the torus and establish a superconcentration formula for the threshold $\textup{T}(g)$ which is the minimal value such that $\{g\geq -\textup{T}(g)\}$ contains a non-contractible loop. This formula follows from a Gaussian Talagrand type inequality.
\end{abstract}

\begin{small}
\tableofcontents
\end{small}
\pagebreak

\section{Introduction}

\subsection{On previous results in Gaussian field percolation and the contributions of the present work}
In the present work, we consider a stationary centered Gaussian field $f$ on $\R^2$ which is a.s. continuous and study the percolation of the excursion sets $\calD_\ell=\{f\geq-\ell\}\subset\R^2$. It is widely believed that, if the field satisfies a few simple assumptions, this percolation model should behave like Bernoulli percolation (see \cite{bs_2002}, \cite{bs_07}, \cite{bg_16}, \cite{bm_18}, \cite{bmw_17}, \cite{bg_17}, \cite{rv_17a}, \cite{rv_17b}, \cite{mv_18}).\\

In \cite{bg_16}, the authors proved an analog of the Russo-Seymour-Welsh (or RSW) theorem (see Lemma 4, Chap. 3 of \cite{bollobas_riordan} or Theorem 11.70 and Equation 11.72 of \cite{Grimmett}) for this percolation model under a few general assumptions which were later generalized in \cite{bm_18}, \cite{rv_17a} and \cite{mv_18}. The authors worked with various versions of the following types of assumptions:
\begin{itemize}
\item (Regularity) The field is a.s. $C^k$ smooth for some adequate $k\geq 1$ (see Condition \ref{cond:regularity} below).
\item (Non-degeneracy) For each $x\in\R^2$, the vector $(f(x),\nabla_xf)$ (or some other vector of derivatives of $f$) is a non-degenerate Gaussian vector.
\item (Symmetry) The field is invariant by some rotation of order greater than two and reflection (see Condition \ref{cond:symmetry} below).
\item (Decay) Either the covariance function $x\mapsto\E\brb{f(0)f(x)}$ (and some of its derivatives), or its convolution square-root (as in Condition \ref{cond:decay} below) decays at a certain speed.
\item (Positivity) The covariance function of the field takes only non-negative values (see Condition \ref{cond:positivity} below). As shown in \cite{pitt_82}, this condition is equivalent\footnote{Actually, this is only shown in finite dimension. For continuous crossing events, one can proceed by approximation as in Appendix A of \cite{rv_17a}.} to the FKG inequality, which is a crucial tool in Bernoulli percolation.
\end{itemize}

The Harris-Kesten theorem and exponential decay of correlations (see \cite{kesten_1980}) were also adapted to this setting, first for the case of the Bargmann-Fock field in \cite{rv_17b}, and later, in an axiomatic setting, in \cite{mv_18}. The setting was similar to the one detailed above except for the fact that the positivity condition, Condition (Weak) \ref{cond:positivity}, was replaced by the strictly stronger Condition (Strong) \ref{cond:positivity}. The proof used in \cite{rv_17b} was inspired by \cite{br_2006}, and used an ad-hoc gaussian version superconcentration inequality for boolean functions called the KKL inequality (due to Kahn, Kalai and Linial, see \cite{kkl_1988}). However, the analogy was too tenuous to be generalized to an axiomatic setting. In contrast, \cite{mv_18} used a randomized algorithm approach inspired by \cite{dcrt_2019}, which was more robust but had its own limits, as we shall see in the next paragraph.\\

The aim of the present paper is twofold. The first, explicit aim is our main result (Theorem \ref{t:sharp_threshold} below), in which we replace the strong positivity assumption: (Strong) Condition \ref{cond:positivity} needed in \cite{mv_18} by the more natural assumption: (Weak) Condition \ref{cond:positivity}, which matches the positivity assumption  made in \cite{bg_16} for the RSW theorem. Secondly, we aimed to provide a proof inspired more by the KKL inequality than randomized algorithms and presented in the native language of smooth Gaussian fields, in order to bring out the underlying mechanism. This takes the form of a superconcentration formula for the (continuous) percolation threshold for Gaussian fields on the torus, inspired by Talagrand's inequality from \cite{talagrand_1994} (see Theorem \ref{t:main_formula} below). In particular, the argument has a very different flavor from that of \cite{mv_18}. Actually, the idea of applying superconcentration formulas on the torus to study the phase transition in planar percolation was introduced in \cite{br_2006}. Finally, we relax some regularity assumptions needed in previous works.

\subsection{Setting, notations and formal statement of the main result}

Throughout this paper, we will work with the white noise representation of the Gaussian field $f$ (as in \cite{mv_18}, for instance). More precisely, given a centered stationary Gaussian field $f$ on $\R^2$ with integrable covariance $\kappa(x)=\E\brb{f(0)f(x)}$, the spectral measure of $f$ is of the form $\varrho(\xi)d\xi$ where $d\xi$ is the Lebesgue measure and $\varrho$ is a continuous, positive valued integrable function. In particular, $\sqrt{\varrho}\in L^2(\R^2)$ so its Fourier transform $q$ belongs to $L^2(\R^2)$. Let $W$ be the $L^2$ white noise on $\R^2$. Then, $q*W$ defines a Gaussian field on $\R^2$ with the law of $f$. In the rest of the paper, unless otherwise stated, we assume that $f$ is of the form
\[
f=q*W
\]
where $W$ is the $L^2$ white noise and $q\in L^2(\R^2)$ satisfies some assumptions among the following:
\begin{condition}[Regularity]\label{cond:regularity}
$ $
\begin{itemize}
\item (Weak) The function $q$ is of class $C^2(\R^2)$.
\item (Strong) The function $q$ is of class $C^3$ and satisfies, for some $\eps>0$ and $C<+\infty$, for all $x\in\R^2$ and for all $\alpha\in\N^2$ such that $|\alpha|\leq 3$, $|\partial^\alpha q(x)|\leq C |x|^{-1-\eps}$. Finally, the support of its Fourier transform contains an open subset.
\end{itemize}
\end{condition}
\begin{condition}[Non-degeneracy]\label{cond:non-degeneracy}
$ $
\begin{itemize}
\item (Weak) The function $q$ is non-zero.
\item (Strong) If $\kappa=q*q$, the following matrix is non-degenerate:
\[
\br{\begin{matrix}
\kappa(0) & ^t\nabla_0\kappa\\
\nabla_0\kappa & \nabla^2_0\kappa
\end{matrix}}\, .
\]
Equivalently, if $f=q*W$ where $W$ is the $L^2$ white noise, the vector $(f(x),\nabla_xf)$ is non-degenerate.
\end{itemize}
\end{condition}
\begin{condition}[Symmetry]\label{cond:symmetry}
The function $q$ is invariant by $\frac{\pi}{2}$-rotations around the origin and by reflections around the axis $\R\times\{0\}$.
\end{condition}
\begin{condition}[Decay]\label{cond:decay}
$ $
There exist $\beta>2$ and $C<+\infty$ such that for each $x\in\R$, $|q(x)|+|\nabla_x q|\leq C|x|^{-\beta}$.
\end{condition}
Note that Condition \ref{cond:decay}, together with continuity, implies that $q\in L^2(\R^2)$.
\begin{condition}[Positivity]\label{cond:positivity}
$ $
\begin{itemize}
\item (Weak) For each $x\in\R^2$, $q*q(x)\geq 0$.
\item (Strong) For each $x\in\R^2$, $q(x)\geq 0$.
\end{itemize}
\end{condition}
\begin{remark}\label{rk:non-dg_is_free}
Consider a stationary Gaussian field $f=q*W$ such that $q$ satisfies the (Weak) Conditions \ref{cond:regularity} and \ref{cond:non-degeneracy}, as well as Condition \ref{cond:decay}. Then it must also satisfy the (Strong) Condition \ref{cond:non-degeneracy}. Indeed, otherwise, by stationarity, $f$ would satisfy a non-trivial differential equation of the form $\lambda f+\mu \partial_vf=0$ which would contradict the decay in pointwise correlations.
\end{remark}
It was shown by Pitt in \cite{pitt_82} that the (Weak) Condition \ref{cond:positivity} is equivalent to a certain form of the FKG inequality (see Lemma \ref{l:FKG} below). The (Strong) version of this condition implies the (Weak) version but is much harder to check. In \cite{mv_18}, the authors relied on the (Strong) version to prove the Harris-Kesten theorem. On the other hand, in \cite{bg_16}, only the (Weak) version of Condition \ref{cond:positivity} was needed to prove the RSW theorem. In Theorem \ref{t:sharp_threshold} below, we prove the Harris-Kesten theorem using only the (Weak) Condition \ref{cond:positivity}. We also replace the (Strong) version of Condition \ref{cond:regularity} by the (Weak), although this improvement is more technical in nature. More precisely, for each $\ell\in\R$, let $\P_\ell$ be the law of $f_\ell:=f+\ell$ so that $\calD_\ell=f_\ell^{-1}(]0,+\infty[)$. For each $R>0$, let $\calR_R:=[0,6R]\times[0,4R]$. We denote by $\textup{Cross}_R$ the event that $\calD_\ell\cap\calR_R$ contains a continuous path connecting $\{0\}\times[0,4R]$ to $\{6R\}\times[0,4R]$. In \cite{mv_18}, the authors prove the following results: 
\begin{theorem}[Exponential decay, see Theorem 1.11  of \cite{mv_18}]\label{t:mv_exponential decay}
Let $f$ be a Gaussian field on $\R^2$ of the form $q*W$ where $W$ is the $L^2$ white noise on $\R^2$ and $q$ satisfies the (Strong) version of Conditions \ref{cond:regularity} and \ref{cond:positivity}, as well as Conditions \ref{cond:symmetry}  \ref{cond:decay} and the (Weak) Condition \ref{cond:non-degeneracy}. Then, for each $\ell>0$, there exists $c=c(\ell)>0$ such that for each $R>0$,
\[
\prob_\ell\brb{\textup{Cross}_R}\geq 1-e^{-cR}\, .
\]
\end{theorem}

\begin{theorem}[The phase transition, see Theorem 1.6  of \cite{mv_18}]\label{t:mv_phase_transition}
Let $f$ be a Gaussian field on $\R^2$ of the form $q*W$ where $W$ is the $L^2$ white noise on $\R^2$ and $q$ satisfies the (Strong) version of Conditions \ref{cond:regularity} and \ref{cond:positivity}, as well as Conditions \ref{cond:symmetry} and \ref{cond:decay}, and the (Weak) Condition \ref{cond:non-degeneracy}. Then, with probability one,
\begin{itemize}
\item for each $\ell>0$, the set $\calD_\ell$ has a unique unbounded connected component.
\item for each $\ell\leq 0$, the set $\calD_\ell$ does not have any unbounded connected components.
\end{itemize}
\end{theorem}

In this article, we further extend the results of \cite{mv_18} to include all fields for which we know the Russo-Seymour-Welsh property to hold (the most general statement so far being Theorem 4.7 of \cite{mv_18}, restated in Lemma \ref{l:rsw} below). More precisely, we establish the following theorem:

\begin{theorem}\label{t:sharp_threshold}
Assume that $q$ satisfies the (Weak) version of Conditions \ref{cond:regularity}, \ref{cond:non-degeneracy}, and \ref{cond:positivity}, as well as Conditions \ref{cond:symmetry} and \ref{cond:decay}. Then, the conclusions of Theorems \ref{t:mv_exponential decay} and \ref{t:mv_phase_transition} both hold.
\end{theorem}

As explained in Remark \ref{rk:non-dg_is_free}, from the assumptions of Theorem \ref{t:sharp_threshold}, $q$ also satisfies (Strong) Condition \ref{cond:non-degeneracy}. We mention this now because (Weak) Condition \ref{cond:regularity}, (Strong) Condition \ref{cond:non-degeneracy} and (Weak) Condition \ref{cond:positivity} are the weakest assumptions known to imply the FKG inequality for continuous crossings (see Lemma \ref{l:FKG}). As shown in \cite{mv_18}, Theorem \ref{t:sharp_threshold} is a consequence of the following proposition:

\begin{proposition}\label{p:phase_transition}
Assume that $q$ satisfies the (Weak) version of Conditions \ref{cond:regularity}, \ref{cond:non-degeneracy} and \ref{cond:positivity}, as well as Conditions \ref{cond:symmetry} and \ref{cond:decay}. Then, for each $\ell>0$,
\[
\lim_{R\rightarrow+\infty}\prob_\ell\brb{\textup{Cross}_R}=1\, .
\]
\end{proposition}

Theorem \ref{t:sharp_threshold} follows from Proposition \ref{p:phase_transition} by a standard renormalization argument which we omit here for brevity. For instance, the argument is given in the proof of Theorem 6.1 of \cite{mv_18} where Proposition \ref{p:phase_transition} is replaced by \cite{mv_18}'s Theorem 5.1.\\

Let us conclude this section by describing the overall layout of the paper. Proposition \ref{p:phase_transition} will follow from Theorem \ref{t:main_formula}, which we will state and establish in Sections \ref{s:torus_proof} and \ref{s:torus_lemmas}. More precisely, in Section \ref{s:torus_proof} we state and prove Theorem \ref{t:main_formula} using a series of technical lemmas which we state along the way. These lemmas are proved in Section \ref{s:torus_lemmas}. Finally, in Section \ref{s:percolation}, we use some percolation arguments to deduce Proposition \ref{p:phase_transition} from Theorem \ref{t:main_formula}.\\

\paragraph{Acknowledgements:} The ideas of this paper stemmed from previous collaborations with Dmitry Beliaev, Stephen Muirhead and Hugo Vanneuville. I am grateful to the three of them for many helpful discussions. I am also thankful to Christophe Garban and Hugo Vanneuville for their comments on a preliminary version of this manuscript.

\section{The key formula: Theorem \ref{t:main_formula}}\label{s:torus_proof}

In this section, we state and prove the main ingredient of the proof of Proposition \ref{p:phase_transition}: Theorem \ref{t:main_formula}. More precisely, the purpose of Subsections \ref{ss:topological_setup} and \ref{ss:main_formula} is to present Theorem \ref{t:main_formula}. In Subsections \ref{ss:talagrand}, \ref{ss:perfect_morse}, \ref{ss:white_noise} and \ref{ss:approximations} we state a series of results used in the proof of Theorem \ref{t:main_formula}. These intermediate results are proved in Section \ref{s:torus_lemmas} below. Finally, in Subsection \ref{ss:main_proof}, we combine the results from previous subsections to prove Theorem \ref{t:main_formula}.\\

\subsection{Admissible events and the threshold map}\label{ss:topological_setup}
Theorem \ref{t:main_formula} is a superconcentration formula for the percolation threshold for certain events which we now describe. Let $\calT=\R/2\pi R_1\Z\times\dots\times\R/2\pi R_d\Z$ be a $d$-dimensional torus equipped with the Lebesgue measure $dx$ inherited from $\R^d$ and let $|\calT|=\int_\calT dx$. For each $l\in\N$, we equip $C^l(\calT)$, the space of $l$-times continuously differentiable, real-valued functions on $\calT$ with the norm $\|u\|_{C^l}:=\max_{|\alpha|\leq l}\max_{x\in\calT}|\partial^\alpha u(x)|$.

\begin{definition}\label{d:admissible_set}
We will say that  $\calA\subset C^0(\calT)$ is \textbf{admissible} if it satisfies the following properties:
\begin{itemize}
\item The set $\calA$ is a \textbf{topological threshold set}: For each $u\in\calA$, and $v\in C^0(\calT)$, if there exists an isotopy $(\phi_t)_{t\in[0,1]}$ of $\calT$ such that $\phi_0=id$ and $\phi_1(\{v\geq 0\})=\{u\geq 0\}$ then, $v\in\calA$.
\item The set $\calA$ is \textbf{increasing}: For each $u\in\calA$ and $v\in C^0(\calT)$, if $u\leq v$ then $v\in\calA$.
\item $\calA$ and $\calA^c$ are both non-empty.
\end{itemize}
\end{definition}
Throughout the rest of the section, we work with a fixed admissible set $\calA$.\\

Let $u\in C^0(\calT)$. Since $\calA$ is increasing and neither $\calA$ nor its complement are empty, there exist $\ell_1,\ell_2\in\R$ such that $u+\ell_1\in\calA^c$ and $u+\ell_2\in\calA$. This allows us to define the \textbf{threshold map}:
\[
\thr:C^0(\calT)\rightarrow\R
\]
that associates to $u$ the infimum of the levels $\ell\in\R$ such that $u+\ell\in\calA$. By construction, this map is Lipshitz on $C^0(\calT)$. Indeed, for each $u,v\in C^0(\calT)$, $|\thr(u)-\thr(v)|\leq\|u-v\|_{C^0(\calT)}$.

\subsection{The concentration formula for $\thr$}\label{ss:main_formula}

Let $\nu\in]0,1[$. Let $W$ be the $L^2$ white noise on $\calT$ and let\footnote{The notation $C^\nu$ is shorthand for $C^{0,\nu}$, the functions of H\"older class $\nu$. See Appendix \ref{s:appendix} for more details.} $q\in C^\nu(\calT)$ and let $f=q*W$. Then, $f$ is an a.s. centered, stationary, continuous Gaussian field on $\calT$ (see Lemma \ref{l:basic_regularity}) . Since $f$ is a.s. $C^0$ and $\thr$ is continuous in $C^0$ topology, $\thr(f)$ is measurable with respect to $f$. The aim of this section will be to prove the following theorem:

\begin{theorem}\label{t:main_formula}
Let $\calA$ be an admissible subset of $C^0(\calT)$. Assume that $\sigma^2=\var(f(x))>0$ (for any $x\in\calT$). The random variable $\thr(f)$ is square integrable and there exists an absolute constant $C<+\infty$ such that
\[
\var\br{\thr(f)}\leq C\sigma^2\brb{1+\left|\ln\br{\frac{\sigma\sqrt{|\calT|}}{\int_\calT |q(x)|dx}}\right|}^{-1}\, .
\]
\end{theorem}
Note that if $q$ is positive valued, then $\int_\calT |q|=\int_\calT q=\br{\int_\calT \kappa}^{1/2}$ so that the argument in the logarithm is just $\sigma\sqrt{\frac{1}{|\calT|}\int_\calT\kappa(x)dx}$. Also, if $\calT$ is the standard torus and $\sigma=1$, Theorem \ref{t:main_formula} shows there exists $C=C(d)<+\infty$ such that
\[
\var\br{\thr(f)}\leq \frac{C}{1+\left|\log(\|q\|_{L^1(\T^d)})\right|}\, .
\]
In particular, if $q$ is very concentrated so that $\|q\|_{L^1}\ll 1$ while $\|q\|_{L^2}=1$, the variance of $\thr(f)$ is small.

\subsection{A Gaussian Talagrand inequality}\label{ss:talagrand}

In order to prove Theorem \ref{t:main_formula}, we will apply the following Gaussian Talagrand inequality to $\thr$.

\begin{theorem}[Gaussian Talagrand inequality, see \cite{cl_2012}]\label{t:talagrand}
Let $\gamma$ be the standard Gaussian measure on $\R^n$. Let $F\in C^\infty_c(\R^n)$. Then,
\begin{equation}\label{e:talagrand}
\textup{Var}_\gamma(F)\leq C\|\nabla F\|_{L^2(\gamma)}^2\max_k\brb{1+\ln\br{\frac{\|\partial_k F\|_{L^2(\gamma)}}{\|\partial_k F\|_{L^1(\gamma)}}}}^{-1}
\end{equation}
where $\textup{Var}_\gamma(F)=\int F(x)^2d\gamma(x)-\br{\int F(x)d\gamma(x)}^2$.
\end{theorem}

Observe that by density of $C^\infty_c$ in $L^p(\gamma)$ Sobolev spaces (for $p=1\text{ and }2$), we can extend this inequality to the case where $F$ is merely Lipschitz on $\R^n$. In particular, if we restrict $\thr$ to a finite dimensional Gaussian space in $C^0(\calT)$, we can indeed apply this inequality to it.

\subsection{Perfect Morse functions and the derivative of $\thr$}\label{ss:perfect_morse}

The derivatives of $\thr$ admit a simple geometric description as long as we restrict $\thr$ to a certain generic class of functions, which we now describe. For each $u\in C^1(\calT)$, let $\crit(u)$ be the set of its critical points. For each $u\in C^2(\calT)$ and each $x\in\crit(u)$, we denote by $H_xu\in\textup{End}(T_x\calT)$ the Hessian of $u$ at $x$. Let $\calM$ be the set of \textbf{perfect Morse functions} on $\calT$, that is, the space of $u\in C^2(\calT)$ such that:
\begin{itemize}
\item For each $x\in\crit(u)$, $H_xu$ is non-degenerate.
\item For each $x,y\in\crit(u)$ distinct, $u(x)\neq u(y)$.
\end{itemize}
This space is open and dense in $C^2(\calT)$. Now, for each $u\in\calM$, by standard Morse theory arguments (see \cite{milnor}), since $\calA$ is admissible (see Definition \ref{d:admissible_set}), the threshold $\thr(u)$ is reached at exactly one critical point of $u$ which we denote by $\sad(u)$. Thus we have defined the \textbf{saddle map}:
\[
\sad:\calM\rightarrow\calT\, .
\]
It turns out that, when restricted to $\calM$, the differential of the map $\thr$ has a simple description:
\begin{lemma}\label{l:threshold}
Let $u\in\calM$ and $v\in C^2(\calT)$. Then\footnote{Our proof shows that the $o$ is uniform in $v$ for $\|v\|_{C^2(\calT)}\leq 1$ but this is not used anywhere.}, as $t\rightarrow 0$, $\thr(u+tv)=\thr(u)+tv(\sad(u))+o(t)$.
\end{lemma}

Combining Theorem \ref{t:talagrand} and Lemma \ref{l:threshold}, we immediately get the following result:

\begin{proposition}\label{p:main}
There exists an absolute constant $C<+\infty$ such the following holds. Let $f$ be a centered Gaussian field with covariance $K$ and a finite dimensional Cameron-Martin space\footnote{See Appendix \ref{s:appendix} for a precise definition of the Cameron-Martin space.} $H\neq\{0\}$. Let $(\psi_k)_k$ be an orthogonal basis of $H$. Assume that $f\in\calM$ a.s. (which implies that $H\subset C^2(\calT)$). Then,
\[
\var(\thr(f))\leq C\E\brb{K\br{\sad(f),\sad(f)}}\max_k\brb{1+\ln\br{ \frac{\sqrt{\E\brb{\psi_k\br{\sad(f)}^2}}}{\E\brb{\left|\psi_k\br{\sad(f)}\right|}}}}^{-1}\, .
\]
\end{proposition}

\subsection{The discrete white noise approximation}\label{ss:white_noise}

To derive Theorem \ref{t:main_formula} from Proposition \ref{p:main}, we must choose the right basis $(\psi_k)_k$. Our choice, which we explain below, is inspired by a similar construction from \cite{mv_18}. For $\eps\in\{\frac{1}{k} \st k\in\N,\ k\geq 1\}=\calE$, let $\Lambda_\eps$ be the set of points with coordinates in $2\pi\eps R_1\Z\times\dots\times 2\pi\eps R_d\Z$ and for each $z\in\Lambda_\eps$, let $Q_z^\eps=z+[0,2\pi\eps R_1]\times\dots [0,2\pi \eps R_d]$. Let $W$ be the $L^2$ white noise on $\calT$. For each $z\in\Lambda_\eps$, let $W^\eps_z=\langle W,\un_{Q_z^\eps}\rangle$. We define the an approximation of $W$ as follows:
\begin{equation}\label{e:approximate_white_noise}
W_\eps:=\sum_{z\in\Lambda_\eps} W_z^\eps\delta_z\, .
\end{equation}
In particular for each function $q\in C^0(\calT)$ and each $x\in\calT$,
\begin{equation}\label{e:approximation_expression}
W_\eps*q(x)=\sum_{z\in\Lambda_\eps} W^\eps_z q(x-z)\, .
\end{equation}
Note that since for each $z,z'\in\Lambda_\eps$ distinct, $\int\un_{Q_z^\eps}(x)\un_{Q_{z'}^\eps}(x)dx=\delta_{z,z'}\eps^d|\calT|$, the collection $(W_z^\eps)_{z\in\Lambda_\eps}$ is a collection of i.i.d. centered normals of variance $\eps^d|\calT|$.
\begin{lemma}\label{l:white_noise_approximation}
Let $l\in\N$ and $\nu\in]0,1[$ and let $q\in C^{l+1,\nu}(\calT)$. Let $f=q*W$ and for each $\eps\in \calE$ $f_\eps=q*W_\eps$. Then, for each $p\in[1,+\infty[$, as $\eps\rightarrow 0$, $\E\brb{\|f_\eps-f\|_{C^l}^p}\rightarrow 0$. In particular, $f_\eps$ converges in law to $f$ in the $C^l$ topology as $\eps\rightarrow 0$.
\end{lemma}
We will actually only use this lemma for $q\in C^\infty(\calT)$. The expression of $f_\eps$ shows that, as long as the family $(q(\cdot-z))_{z\in\Lambda_\eps}$ is independent, it forms an orthogonal basis of its Cameron-Martin space (see Lemma \ref{l:CM_basis}). In the following lemma, we check that this condition is generic.

\begin{lemma}\label{l:generic_independence}
The set of functions $q\in C^\infty(\calT)$ such that for each $\eps\in\calE$, the family $(q(\cdot-z))_{z\in\Lambda_\eps}$ is independent, is dense in $C^\infty(\calT)$.
\end{lemma}

\subsection{General approximation arguments}\label{ss:approximations}

To prove Theorem \ref{t:main_formula}, we apply Proposition \ref{p:main} to approximations of the field $f$. It will be useful to know that $\sad$ behaves well under approximations:

\begin{lemma}\label{l:saddle_continuity}
The map $\sad:\calM\rightarrow\calT$ is $C^2$-continuous. In particular, if $(f_i)_i$ is a sequence of a.s. $C^2$ Gaussian fields converging in law in the $C^2$ topology to a Gaussian field $f$ such that both $f$ and each $f_i$ are a.s. in $\calM$, then, $\sad(f_i)$ converges in law to $\sad(f)$ as $i\rightarrow+\infty$.
\end{lemma}

In order to apply this lemma, we need to define a class of fields which is both generic and stable, whose elements are a.s. perfect Morse functions. For each $l\in\N$, let $\Gamma_l(\calT)$ be the space of Gaussian measures on $C^l(\calT)$ equipped with the topology of weak-* convergence in $C^l$ topology.
\begin{definition}
Let $\Gamma_3^{nd}(\calT)$ be the set of $\gamma\in\Gamma_3(\calT)$ such that if $f$ is an a.s. $C^3$ Gaussian field on $\calT$ with measure $\gamma$, then $f$ satisfies the following properties:
\begin{itemize}
\item For each $x,y\in\calT$ distinct, the following vector is non-degenerate $(f(x),\nabla_xf,f(y),\nabla_yf)$.
\item For each $x\in\calT$, the following vector is non-degenerate $(f(x),\nabla_xf,H_xf)$.
\end{itemize}
\end{definition}
If the law of $f$ belongs to $\Gamma_3^{nd}(\calT)$, then $\sad(f)$ is a.s. well defined:

\begin{lemma}\label{l:nd_is_morse}
For all $\gamma\in\Gamma_3^{nd}(\calT)$, $\gamma(\calM)=1$.
\end{lemma}

Moreover, the class $\Gamma_3^{nd}(\calT)$ is both open and dense in $\Gamma_3(\calT)$:

\begin{lemma}\label{l:nd_is_open}
The subset $\Gamma_3^{nd}(\calT)$ is open in $\Gamma_3(\calT)$.
\end{lemma}

\begin{lemma}\label{l:nd_dense_for_q}
Let $\calC$ be the set of functions $q\in C^\infty(\calT)$ such that the Gaussian measure induced by $q*W$ in $\Gamma_3(\calT)$ belongs to $\Gamma_3^{nd}(\calT)$. Then, for each $\nu\in]0,1[$, $\calC$ is dense in $C^\nu(\calT)$.
\end{lemma}

\subsection{Proof of Theorem \ref{t:main_formula}}\label{ss:main_proof}

In this subsection, we combine the results of Subsections \ref{ss:talagrand}, \ref{ss:perfect_morse}, \ref{ss:white_noise} and \ref{ss:approximations} to prove Theorem \ref{t:main_formula}.
\begin{proof}[Proof of Theorem \ref{t:main_formula}]
Let $\gamma$ be the law of $f$. We start by assuming that $q\in C^\infty(\calT)$ so that in particular $\gamma\in\Gamma_3(\calT)$. We also make the assumption that $\gamma\in\Gamma_3^{nd}(\calT)$ and that for each $\eps\in\{\frac{1}{k} \st k\in\N,\ k\geq 1\}=\calE$, $(q(\cdot-z))_{z\in\Lambda_\eps}$ is linearly independent (as in Lemma \ref{l:generic_independence}). For each $\eps\in\calE$, let $\gamma_\eps$ be the law of $f_\eps=q*W_\eps$ where $W_\eps$ is the approximation of the $L^2$ white noise introduced in \eqref{e:approximate_white_noise} and let $K_\eps$ be the covariance of $f_\eps$. In the rest of the proof, the symbol $\eps$ will denote an element of $\calE$. By Lemma \ref{l:white_noise_approximation}, as $\eps\rightarrow 0$, $\gamma_\eps$ converges to $\gamma$ in the weak-* topology over $C^3(\calT)$. By Lemmas \ref{l:nd_is_open} and \ref{l:nd_is_morse}, for all small enough $\eps>0$, $f_\eps$ is a.s. Morse. In particular, $\sad(f_\eps)$ and $\sad(f)$ are well defined. By Lemma \ref{l:CM_basis}, since $(q(\cdot-z))_{z\in\Lambda_\eps}$ is linearly independent, it is an orthogonal basis of the Cameron-Martin space of $f_\eps$. Moreover, since $f_\eps$ is $\Lambda_\eps$-periodic, the law of $q(\sad(f)-z)$ does not depend on $z\in\Lambda_\eps$. In particular, by Proposition \ref{p:main},
\begin{equation}\label{e:skeleton_1}
\var(\thr(f_\eps))\leq C\E\brb{K_\eps\br{\sad(f_\eps),\sad(f_\eps)}}\brb{1+\ln\br{ \frac{\sqrt{\E\brb{q\br{\sad(f_\eps)}^2}}}{\E\brb{\left|q\br{\sad(f_\eps)}\right|}}}}^{-1}\, .
\end{equation}
By Lemmas \ref{l:white_noise_approximation} and \ref{l:cv_gamma_to_k}, $(K_\eps)_\eps$ converges uniformly to $K$. This observation, combined with Lemma  \ref{l:saddle_continuity}, shows that, as $\eps\rightarrow 0$, the right-hand side of \eqref{e:skeleton_1} converges to the same quantity with $K_\eps$ and $f_\eps$ replaced by $K$ and $f$ respectively. Moreover, by Lemma \ref{l:white_noise_approximation}, $\lim_{\eps\rightarrow 0}\E\brb{\|f_\eps-f\|_{C^0}^2}=0$. Since $\thr$ is $C^0$-Lipschitz, we have $\lim_{\eps\rightarrow 0}\var\br{\thr(f_\eps)}=\var\br{\thr(f)}$. Hence, taking $\eps\rightarrow 0$ in \eqref{e:skeleton_1} yields
\[
\var\br{\thr(f)}\leq C\E\brb{K\br{\sad(f),\sad(f)}}\brb{1+\ln\br{ \frac{\sqrt{\E\brb{q\br{\sad(f)}^2}}}{\E\brb{\left|q\br{\sad(f)}\right|}}}}^{-1}\, .
\]
Now, since $f$ is stationary, $\sad(f)$ is uniformly distributed on the torus. Therefore,
\begin{equation}\label{e:skeleton_2}
\var\br{\thr(f)}\leq \frac{C}{|\calT|}\int_\calT K(x,x)dx\brb{1+\ln\br{ \frac{\sqrt{\|q\|_{L^2}|\calT|}}{\|q\|_{L^1}}}}^{-1}\, .
\end{equation}
But for each $x\in\calT$, $K(x,x)=q*q(0)=\sigma^2$ as announced. Let us now lift the assumptions on $q$. We now only assume that $q\in C^\nu(\calT)$ for some $\nu\in]0,1[$. By Lemmas \ref{l:generic_independence} and \ref{l:nd_dense_for_q}, we may find a sequence $(q_n)_{n\in\N}$ of smooth $C^\infty$ functions converging in $C^\nu$ to $q$ such that for each $n$, the measure of $q_n*W$ belongs to $\Gamma_3^{nd}(\calT)$ and such that for each $\eps\in\calE$ $(q_n(\cdot-z))_{z\in\Lambda_\eps}$ is independent. All the terms on the right-hand side of \eqref{e:skeleton_2} are obviously $C^\nu$ continuous in $q$. For the term $\var\br{\thr(\cdot)}$, since $\thr$ is $C^0$-Lipschitz, it is enough to show that $\lim_{n\rightarrow+\infty}\E\brb{\|(q_n-q)*W\|_{C^0}^2}=0$. By Lemma \ref{l:sup_lemma}, it is enough to show that $\lim_{n\rightarrow 0}\|(q_n-q)*(q_n-q)\|_{C^\nu}=0$. But this follows from Lemma \ref{l:sup_convolution}. Thus, taking $n\rightarrow+\infty$, the proof is over.
\end{proof}

\section{Proof of the lemmas from Section \ref{s:torus_proof}}\label{s:torus_lemmas}

In this section, we prove the results stated in Subsections \ref{ss:perfect_morse}, \ref{ss:white_noise} and \ref{ss:approximations}. Subsections \ref{ss:topology_lemmas}, \ref{ss:white_noise_lemmas} and \ref{ss:approx_lemmas} are mutually independent and only use results from Appendix \ref{s:appendix}.

\subsection{Proof of Lemmas \ref{l:threshold} and \ref{l:saddle_continuity}}\label{ss:topology_lemmas}

The proof of Lemmas \ref{l:threshold} and \ref{l:saddle_continuity} relies on the following lemma:
\begin{lemma}\label{l:path}
Let $u\in\calM$. Let $x_0\in\textup{Crit}(u)$ and set $u(x_0)=:a$. Then, there exists $\delta=\delta(u)>0$ such that for each $v\in C^2(\calT)$ with $\|v\|_{C^2}\leq 1$ and each $t\in]-\delta,\delta[$, there exist $x_t\in\calT$ and $a_t\in\R$ such that the following hold:
\begin{itemize}
\item For each $t\in]-\delta,\delta[$, $u+tv\in\calM$.
\item For each $t\in]-\delta,\delta[$, $x_t$ is a critical point of $u+tv$ with critical value $a_t$.
\item The family $(x_t)_t$ is continuous at $0$, uniformly in $v$: for each $\eps>0$, there exists $\tilde{\delta}=\tilde{\delta}(u,\eps)\in]0,\delta]$ such that for each $t\in]-\tilde{\delta},\tilde{\delta}[$, $|x_t-x_0|$.
\item The family $(a_t)_t$ is differentiable at $0$ and $\frac{da_t}{dt}\big|_{t=0}=v(x_0)$.
\item There exists $W=W(u)\subset\calT$ an open neighborhood of $x_0$ such that for all $t\in]-\delta,\delta[$, $x_t\in W$ and $u+tv$ has no other critical points in $W$.
\end{itemize}
\end{lemma}
\begin{proof}
Since the statement of the lemma is local, we may assume that $v$ is supported near $x_0$ and work in local charts. By simplicity, we assume that $x_0=0$ and $a=0$. Since $u\in\calM$, the Hessian $H_0u$ of $u$ at $0$ is non-degenerate. There exist $W_1=W_1(u)\subset\R^2$ a neighborhood of $0$, $B=B(u)\subset\textup{GL}_d(\R)$ a convex neighborhood of $H_0u$ and $\delta_1=\delta_1(u)>0$ such that for each $v\in C^2_c(W_1)$ with $\|w\|_{C^2}\leq \delta_1$ and each $x\in W$, $H_x(u+w)\in B$. Let us fix $v\in C^2_c(W_1)$ with $\|v\|_{C^2}=1$ and let $t\in]-\delta_1,\delta_1[$. By construction, $x\mapsto \nabla_x(u+tv)$ is a local diffeomorphism at $0$ and $\nabla_0u=0$. Thus, there exists $W_2=W_2(u)\subset W_1$ an open neighborhood of $0$ and $\delta_2=\delta_2(u)\in]0,\delta_1]$ such that for each $t\in]-\delta_2,\delta_2[$, $u+tv$ has exactly one critical point in $W_2$, which we denote by $x_t$, which is differentiable in $t$ uniformly in $v$ (in particular, this proves the third poitn of th lemma). Therefore, letting $a_t=(u+tv)(x_t)$, we get $\frac{d a_t}{dt}\big|_{t=0}=\nabla_0u(\dot{x}_0)+v(0)=v(0)$. Taking $\delta$ to be the infimum of the $\delta_2$ for all the critical points of $u$ ensures that $u+tv\in\calM$ for $t\in]-\delta,\delta[$.
\end{proof}
We will now simultaneously prove Lemmas \ref{l:threshold} and \ref{l:saddle_continuity}.
\begin{proof}[Proof of Lemmas \ref{l:threshold} and \ref{l:saddle_continuity}]
Let $u\in\calM$ and let $\delta=\delta(u)>0$ given by Lemma \ref{l:path}. Let $v\in C^2(\calT)$ with $\|v\|_{C^2}\leq 1$ and let $x_0=\sad(u)$. Let $(a_t)_t$ and $(x_t)_t$ be the paths given by Lemma \ref{l:path} associated to $u$, $x_0$ and $v$. By continuity of $\thr$, we must have $\thr(u+tv)=a_t$, which in turn implies that for each $t\in]-\delta,\delta[$, $\sad(u+tv)=x_t$. Thus, by Lemma \ref{l:path} $\frac{d}{dt}\big|_{t=0}\thr(u+tv)=v\br{\sad(u)}$ and $\sad$ is continuous at $u$: for each $\eps>0$, there exists $\delta_1\in]0,\delta]$ independent of $v$ such that for all $t\in]-\delta_1,\delta_1[$, $\left|\sad(u+tv)-\sad(u)\right|\leq\eps$.
\end{proof}

\subsection{Proof of Lemmas \ref{l:white_noise_approximation} and \ref{l:generic_independence}}\label{ss:white_noise_lemmas}

\begin{proof}[Proof of Lemma \ref{l:white_noise_approximation}]
By induction it is enough to prove the case where $l=0$. To prove the convergence we use Lemma \ref{l:sup_lemma} to show that the expected supremum of $\|f-f_\eps\|_{C^0}$ tends to $0$ as $\eps\rightarrow+\infty$. To apply the lemma we assume that for $i=1,\dots,d$, $R_i\geq 1$, as we may by rescaling the torus by an adequate factor. For each $l>0$, this will only  multiply all convolutions by the same factor. For each $\eps\in\calE$, let $K_\eps$ be the covariance function of $f-f_\eps$. Lemma \ref{l:sup_lemma} reduces the proof to showing that the function $h_\eps(x)=K_\eps(x,x)$ converges to $0$ in $C^\nu$-norm. First of all, for each $x\in\calT$,
\[
q*W(x)-q*W_\eps(x)=\sum_{z\in\Lambda\eps}\langle W,(q(x-\cdot)-q(x-z))\un_{Q^\eps_z}\rangle\, .
\]
so that
\[
h_\eps(x)=\sum_{z\in\Lambda\eps}\int_{Q^\eps_z}(q(x-y)-q(x-z))^2dy\leq |\calT|\|q\|_{C^1}^2\eps^{2}\xrightarrow[\eps\to 0]{}0\, .
\]
Now, take $x,x'\in\calT$. Then, for each $z\in\Lambda_\eps$,
\begin{align*}
\left|q(x-y)-q(x-z)-q(x'-y)+q(x'-y)\right|&\leq \int_0^1\left|\langle\nabla_{(x-(ty+(1-t)z)}q-\nabla_{(x'-(ty+(1-t)z)}q,y-z\rangle\right|dt\\
&\leq \|q\|_{C^{1,\nu}}\textup{dist}(x,x')^\nu\eps
\end{align*}
so that, using $|a^2-b^2|=|a-b||a+b|$, we get
\[
\left|(q(x-y)-q(x-z))^2-(q(x'-y)-q(x-z))^2\right|\leq 4\|q\|_{C^{1,\nu}}^2\textup{dist}(x,x')^\nu\eps\, .
\]
But this last relation implies that
\begin{align*}
|h_\eps(x)-h_\eps(x')|&\leq\sum_{z\in\Lambda_\eps}\int_{Q^\eps_z}\left|(q(x-y)-q(x-z))^2-(q(x'-y)-q(x-z))^2\right|dy\\
\\
&\leq 4|\calT|\|q\|_{C^{1,\nu}}^2\textup{dist}(x,x')^\nu\eps\xrightarrow[\eps\to 0]{}0\, .
\end{align*}
Thus, $h_\eps \xrightarrow[\eps\to 0]{}0$ in $C^\nu$. By Lemma \ref{l:sup_lemma}, $\E\brb{\|f_\eps-f\|_{C^0}}\xrightarrow[\eps\to 0]{}0$ and $(\|h\|_{C^\nu}^{1/2}\|f_\eps-f\|_{C^0})_{\eps>0}$ has bounded Gaussian tails so in particular the convergence takes place in $L^p$ for all $p\in[1,+\infty[$.
\end{proof}

\begin{proof}[Proof of Lemma \ref{l:generic_independence}]
For each $\eps\in\calE$, let $\calG_\eps$ be the set of functions $q\in C^\infty(\calT)$ such that $\det(q(z-z')_{z,z'\in\Lambda_\eps})\neq 0$. Note that $\cap_{\eps\in\calE}\calG_\eps$ is a subset of the set of functions under consideration. It is therefore enough to show that for each $\eps\in\calE$, $\calG_\eps$, is open and dense. The fact that it is open is clear from the definition. It is dense because the map $q\mapsto (q(z-z'))_{z,z'\in\Lambda_\eps}$ is both continuous and open.
\end{proof}

\subsection{Proof of Lemmas \ref{l:nd_is_morse}, \ref{l:nd_is_open} and \ref{l:nd_dense_for_q}}\label{ss:approx_lemmas}
We begin with the proof of Lemma \ref{l:nd_is_morse}.
\begin{proof}[Proof of Lemma \ref{l:nd_is_morse}]
Let $f$ be an a.s. $C^3$ Gaussian field on $\calT$ with law $\gamma\in\Gamma_3^{nd}(\calT)$. We will show that $f\in\calM$ almost surely. This will follow by applying Lemma 11.2.10 of \cite{adler_taylor} in the right setting. First, notice that $(\nabla_xf,\det(H_xf))_{x\in\calT}$ defines an a.s. $C^1$ field on $\calT$ with values in $\R^{d+1}$ and with uniformly bounded pointwise density. By Lemma 11.2.10 of \cite{adler_taylor}, a.s. it does not vanish on $\calT$. Thus, $f$ is a.s. a Morse function on $\calT$. Applying a similar reasoning to the field $(x,y)\mapsto (f(x)-f(y))^2+|\nabla_xf|^2+|\nabla_yf|^2$ on a compact exhaustion of $\calT\times\calT\setminus\{(x,x) \st x\in\calT\}$, we see that a.s., $f$ does not have two critical points at the same height. Hence, a.s., $f\in\calM$.
\end{proof}
For the proofs of Lemmas \ref{l:nd_is_open} and \ref{l:nd_dense_for_q}, we will use the following characterization of $\Gamma_3^{nd}(\calT)$.
\begin{lemma}\label{l:nd_basics}
Let $f$ be a centered Gaussian field on $\calT$ with law $\gamma\in\Gamma_3(\calT)$. Let $H$ be its Cameron-Martin space and let $K$ be the covariance function of the field $f$. For each $x,y\in\calT$ distinct let $P^1_{x,y}:C^1(\calT)\rightarrow\R^{2d+2}$ be defined as follows:
\[
P^1_{x,y}u=(u(x),\nabla_xu,u(y),\nabla_yu)\, .
\]
For each $x\in\calT$, let $P^2_x:C^2(\calT)\rightarrow\R^{d+d(d+1)/2}$ be defined as follows:
\[
P^2_xu=(\nabla_xu,H_xu)
\]
where the space of $d\times d$ symmetric matrices is identified with $\R^{d(d+1)/2}$. Then, the following assertions are equivalent:
\begin{enumerate}
\item The measure $\gamma$ belongs to $\Gamma_3^{nd}(\calT)$.
\item For each $x,y\in\calT$ distinct, the vectors $P^1_{x,y}f$ and $P^2_xf$ are non-degenerate.
\item For each $x,y\in\calT$ distinct, the matrices $(P^1_{x,y}\otimes P^1_{x,y}) K$ and $(P^2_x\otimes P^2_x)K$ are non-degenerate.
\item For each $x,y\in\calT$ distinct, the maps $P^1_{x,y}$ and $P^2_x$ are surjective when restricted to $H$.
\end{enumerate}
\end{lemma}
\begin{proof}
The first two points are equivalent by definition of $\Gamma_3^{nd}(\calT)$. The second and third point are equivalent because the covariance matrix of $P^1_{x,y}f$ is $(P^1_{x,y}\otimes P^1_{x,y})K$ and similarly for $P^2_xf$. Now, assume that the third assertion is true and fix $x,y\in\calT$ distinct. Let $P=P^1_{x,y}$ (resp. $P=P^2_x$) and $k=2d+2$ (resp. $k=d+d(d+1)/2$). Then, for each $\lambda\in(\R^k)^*$, $(\lambda\circ P\otimes id)K\in H$ because $K$ is the reproducing kernel of $H$. Since the matrix $(P\otimes P)K$ is non-degenerate, this means that for each $Y\in\R^k$, we can find $\lambda\in(\R^k)^*$ such that $P((\lambda\circ P\otimes id)K)=(\lambda\circ P\otimes P)K=Y$. Thus, $P:H\rightarrow\R^k$ is surjective and the fourth assertion holds. Finally, let us assume that the fourth assertion holds. Let $H_P$ be the kernel of $P$ in $H$ and let $H_P^\perp$ be its orthogonal. Then, $f$ can be written as an independent sum $f_1+f_2$ where $H_P$ is the Cameron-Martin space of $f_1$ and $H_P^\perp$ is that of $f_2$. Thus, $Pf=Pf_2$. But $f_2\in H_P^\perp$, which has dimension $k$ and on which $P$ is injective. Moreover, $f_2$ defines a non-degenerate Gaussian vector in $H_P^\perp$. In particular, $Pf_2$ defines a non-degenerate Gaussian vector in $\R^k$ and the second assertion is true.
\end{proof}
\begin{proof}[Proof of Lemma \ref{l:nd_is_open}]
Throughout the proof, we use the characterization of $\Gamma_3^{nd}(\calT)$ given by the second point of Lemma \ref{l:nd_basics}. Let $\gamma\in\Gamma_3(\calT)$ and let $(\gamma_n)_{n\in\N}$ be a sequence of measures in $\Gamma_3(\calT)\setminus\Gamma_3^{nd}(\calT)$. Let us show that $\gamma\notin\Gamma_3^{nd}(\calT)$. For each $n\in\N$ let $f_n$ have law $\gamma_n$ and let $f$ have law $\gamma$. Up to extracting subsequences one can assume either that there exists a sequence $(x_n)_{n\in\calT}$ such that $P^2_{x_n}f_n$ is degenerate for each $n$, or that there exists a sequence $(x_n,y_n)_{n\in\N}$ of pairs of distinct points such that $P^1_{x_n,y_n}f_n$ is degenerate for all $n$. We treat the second case since it is more complex. The first follows a simpler version of the same argument. For each $n\in\N$, there exist two distinct points $x_n,y_n\in\calT$ and two linear forms $\lambda_n,\lambda_n'\in(\R^{d+1})^*$ such that, a.s., $(\lambda_n,-\lambda_n')P_{x_n,y_n}^1 f_n=0$. To begin with, by compactness, we may assume (up to extracting subsequences), that $(x_n,y_n)_n$ converges to some $(x,y)\in\calT$ so that $P^1_{x,y}f$ is degenerate. If $x\neq y$ then we have just proved that $\gamma\notin\Gamma_3^{nd}(\calT)$ so we are done. Assume now that $x=y$. Again, by compactness, we may assume that:
\begin{itemize}
\item If $\eps_n=|x_n-y_n|$, $\eps_n^{-1}(x_n-y_n)=\tau_n$ converges, as $n\rightarrow+\infty$, to some $\tau\in S^{d-1}$.
\item If $\eta_n=|\lambda_n-\lambda_n'|$, $\eta_n^{-1}[\lambda_n-\lambda_n')=\varpi_n$ converges, as $n\rightarrow+\infty$, to some $\varpi\in (\R^{d+1})^*$.
\end{itemize}
For each $n\in\N$, let $a_n=\max(\eps_n,\eta_n)$. Define $P_x^1f:=(f(x),\nabla_xf)$. We have a.s., for each $n\in\N$,
\[
\br{\lambda_n-\lambda_n'} P_{x_n} f_n + \lambda_n'\br{P_{x_n} f_n-P_{y_n}f_n}=0
\]
so that, by applying a Taylor expansion around $y_n$ and using the continuity in $x$ of $\nabla P^1_xf$, we get
\[
(\eta_n/a_n)\br{\varpi(P_x^1f_n)+o(1)}+(\eps_n/a_n)\br{\partial_\tau P^1_xf+o(1)}=0\, .
\]
Taking $n\rightarrow+\infty$, we see that there exist $\alpha,\beta\in\R$ such that, a.s.,
\[
\alpha\varpi P^1_xf+\beta \partial_\tau P^1_xf=0\, .
\]
In particular, this implies that $P^2_xf$ is degenerate so, using Assertion 2 from Lemma \ref{l:nd_basics}, we have $\gamma\in\Gamma_3(\calT)\setminus\Gamma_3^{nd}(\calT)$ as announced.
\end{proof}

\begin{proof}[Proof of Lemma \ref{l:nd_dense_for_q}]
First of all, $C^\infty(\calT)$ is dense in $C^\nu(\calT)$. Now let $q\in C^\infty(\calT)$. Let $\hat{\calT}=R_1^{-1}\Z\times\dots\times R_d^{-1}\Z$ be the dual lattice of $\calT$ and for each $u\in C^\infty(\calT)$ and each $w\in\hat{\calT}$, let $c_w(u)=\frac{1}{\sqrt{|\calT|}}\int_\calT u(x)e^{i\langle w,x\rangle}dx$ be the $w$-th Fourier coefficient of $u$. For all $R>0$, let $B_R=B_{eucl}(0,R)\cap\hat{\calT}$ and let $L^2_R(\calT)$ be the space of functions $u\in L^2(\calT)$ such that $c_w(u)=0$ for $w\notin B_R$. By standard Fourier analysis arguments, for any $u\in C^\infty(\calT)$, the $L^2$-orthogonal projection of $u$ onto $L_R^2(\calT)$ converges to $u$ in $C^\infty$ as $R\rightarrow+\infty$ and if $u$ is real-valued, so is its projection. Next, observe that the projection onto $L^2_R(\calT)$ of $u$ can itself be approximated in $C^\infty$ by (real valued) functions $\tilde{u}$ such that the support of $(c_w(\tilde{u}))_w$ is exactly $B_R$. Finally, observe that if $q\in L^2_R(\calT)$ is such that the support of $(c_w(q))_w$ is $B_R$, then $q*W$ is a random linear combination of cosines and sines of $\langle w,\cdot\rangle$ where the coefficients are independent centered normals whose variance is positive exactly when $w\in B_R$. In particular, by Lemma \ref{l:CM_basis}, the Cameron-Martin space of $q$ is $L^2_R(\calT)$ (although equipped with a scalar product depending on $q$). Hence, using Assertion 4 of Lemma \ref{l:nd_basics}, the problem is now reduced to showing that for all large enough $R>0$, the space $H=L^2_R(\calT)$ satisfies the fourth assertion in Lemma \ref{l:nd_basics} (which is independent of the scalar product on it!). By Lemma \ref{l:nd_is_open} the non-degeneracy condition is open in $\Gamma_3(\calT)$. Since $\cup_{R>0}L^2_R(\calT)$ is dense in $C^\infty(\calT)$, it is enough to find any finite-dimensional space $H\subset C^\infty(\calT)$ on which the maps $P^1_{x,y}$ and $P^2_x$ for $x\neq y$ are surjective (indeed, one can then approximate elements of a basis of $H$ by elements of $\cup_{R>0}L^2_R(\calT)$ which will yield an approximation of the measures in $\Gamma_3(\calT)$). But this follows from the multijet transversality theorem. Indeed, let $m$ be a (large) integer and let $F:\calT\rightarrow\R^m$ be a $C^\infty$ smooth map and let $H_F$ be the space generated by the coordinates $F_i:\calT\rightarrow\R$ for $i\in\{1,\dots,m\}$. The condition that $H_F$ does not satisfy Assertion 4 from Lemma \ref{l:nd_basics} has codimension arbitrarily large in $C^\infty(\calT,\R^m)$ when $m\rightarrow+\infty$. In particular, for large enough values of $m$, the multijet transversality theorem (see Theorem 4.13, Chapter 2 of \cite{golubitsky_guillemin}) applied to the multijet
\begin{align*}
\{(x,y,z)\in\calT^3 \st x\neq y\}&\rightarrow\R^{2d+2+d+d(d+1)/2}\\
(x,y,z)&\mapsto (F(x),F(y),d_xF,d_yF,d_zF,d^2_zF)
\end{align*}
the set of $F$ such that $H_F$ satisfies Assertion 4 of Lemma \ref{l:nd_basics} is dense. In particular it is non-empty so the proof is over for the case of $\calC$.
\end{proof}

\section{Proof of Proposition \ref{p:phase_transition}}\label{s:percolation}

In the present section, we prove Proposition \ref{p:phase_transition} using Theorem \ref{t:main_formula}. The proof should be reminiscent of the strategy used in \cite{br_2006} for Bernoulli percolation. Here, since Definition \ref{d:admissible_set} is a bit restrictive, we cannot completely follow the strategy of \cite{br_2006}. Instead, we study loop percolation events, which are topological, and use them to detect crossings of rectangles on the torus. In Subsection \ref{ss:loop_corollary} we extract a loop percolation estimate from Theorem \ref{t:main_formula}, in Subsection \ref{ss:percolation_estimates} we establish some elementary percolation estimates and finally, in Subsection \ref{ss:the_conclusion}, we combine the results of the two previous subsections to prove Proposition \ref{p:phase_transition}.\\

Throughout this section, we consider a square torus $\calT_R$ of dimension $d=2$ with length $100 R$ for some $R>0$. Given a rectangle of the form $\calR=[a,b]\times[c,d]\subset\R^2$ (or $\calT_R$). We call $\{a\}\times [c,d]$ its \textbf{left side} and $\{b\}\times [c,d]$ its \textbf{right side}. We say that a subset $A\subset\R^2$ (or $A\subset\calT_R$) contains a \textbf{crossing} of $\calR$ from left to right if there exists a continuous path in $A\cap\calR$ joining the left and right hand sides of the rectangle.

\subsection{An application of Theorem \ref{t:main_formula}}\label{ss:loop_corollary}

In the proof of Proposition \ref{p:phase_transition}, we will use the following corollary of Theorem \ref{t:main_formula}:
\begin{corollary}\label{c:loop_percolation}
Let $d\in\N$, $d\geq 2$. Let $\calT=(\R/L\Z)^d$ for some fixed $L>0$ so that $H_1(\calT)$ is canonically isomorphic to $\Z^d$. Let $q\in L^2\br{\calT}$ satisfy (Weak) Condition \ref{cond:regularity}, (Strong) Condition \ref{cond:non-degeneracy} and Condition \ref{cond:symmetry}. Let $\sigma^2=\int_{\calT}q^2(x)dx>0$. Let $W$ be the $L^2$ white noise on $\calT$ and let $f=q*W$. Let $\calL$ be the set of $u\in C^0(\calT)$ such that there exists a smooth loop $\gamma: S^1\rightarrow u^{-1}([0,+\infty[)$ whose homology class in $(n_1,\dots,n_d)\in H_1(\calT)\simeq\Z^d$ satisfies $n_1>0$. Let
\begin{equation}\label{e:alpha}
\alpha(q)=\brb{1+\left|\ln\br{\frac{\|q\|_{L^2(\calT)}}{\|q\|_{L^1(\calT)}}\sqrt{|\calT|}}\right|}^{-1/2}\, .
\end{equation}

Assume first that $d=2$. Then, there exists a universal constant $C<+\infty$ such that for each $\eps>C\alpha(q)$,
\[
\prob\left[f-\sigma\eps\in\calL\right]\leq C\eps^{-2}\alpha(q)^2\, .
\]
Moreover, if $d\geq 2$, then, for each $\eps\geq C\alpha(q)$,
\[
\prob\left[f+\sigma\eps\notin\calL\right]\leq C\eps^{-2}\alpha(q)^2\, .
\]
\end{corollary}
Clearly, the set $\calL$ is admissible in the sense of Definition \ref{d:admissible_set} so we may apply Theorem \ref{t:sharp_threshold} to it. This theorem is a variance bound. In order to obtain a concentration result such as Corollary \ref{c:loop_percolation}, we need some control on the quantiles of the threshold functional. To this end, in the following lemma, we first study the homology of the excursion sets of (deterministic) functions on $\T^d$. Its proof is an elementary exercise in algebraic topology, but we include it below for completeness.
\begin{lemma}\label{l:algebraic_topology}
Let $f:\T^d\rightarrow\R$ be a smooth function with no critical points at height $0$. Let $A=f^{-1}(]0,+\infty[)$ and $B=f^{-1}(]-\infty,0[)$. Then, there exists a smooth loop $\gamma: S^1\rightarrow A\cup B$ that is non-contractible. More precisely, the following holds. Let $\iota^+:A\rightarrow\T^2$ and $\iota^-:B\rightarrow\T^2$ be the inclusion maps and let $\iota^\pm_*$ be the induced maps in the $H_1$-singular homology.
\begin{enumerate}
\item For any two distinct coordinates $i,j\in\{1,\dots,d\}$, the image of the map $(\iota^+_*\oplus\iota^-_*):(\sigma_1,\sigma_2)\mapsto \iota^+_*(\sigma_1)+\iota^-_*(\sigma_2)$ contains a vector $(n_1,\dots,n_d)\in H_1(\T^d)\simeq\Z^d$ such that $(n_i,n_j)\neq 0$.
\item Assume that $d=2$. Then, one of the three possible assersions holds:
\begin{itemize}
\item The images of $\iota^+_*$ and $\iota^-_*$ are both isomorphic to $\Z$ as $\Z$-modules and are $\R$-colinear.
\item The map $\iota^+_*$ is surjective while the map $\iota^-_*$ is zero.
\item The map $\iota^-_*$ is surjective while the map $\iota^+_*$ is zero.
\end{itemize}
\end{enumerate}
\end{lemma}
Using this lemma, by symmetry and duality arguments, at the end of this subsection, we will deduce the following estimate:
\begin{lemma}\label{l:loops_and_duality}
Assume that $d\geq 2$. Let $f=q*W$  where $q\in L^2(\T^d)$ satisfies (Weak) Condition \ref{cond:regularity}, (Strong) Condition \ref{cond:non-degeneracy} and Condition \ref{cond:symmetry} and $W$ is the $L^2$ white noise on $\T^d$. Recall the set $\calL$ defined in Corollary \ref{c:loop_percolation} and the threshold map $\textup{T}_\calL$ defined in Subsection \ref{ss:topological_setup}. Then,
\[
\prob\left[T_\calL(f)\leq 0\right]\geq 1/4\, .
\]
Moreover, if $d=2$,
\[
\prob\left[T_\calL(f)\geq 0\right]\geq 1/4\, .
\]
\end{lemma}
Given Theorem \ref{t:main_formula} and Lemma \ref{l:loops_and_duality}, Corollary \ref{c:loop_percolation} is a direct application of the following elementary lemma.
\begin{lemma}\label{l:chebychev_bound}
Let $X$ be a real random variable with finite variance $a^2$ such that $\prob[X\leq 0]=t^2>0$. Then, for each $\ell>2(a/t)$,
\[
\prob[X\geq \ell]\leq 4a^2\ell^{-2}\, .
\]
\end{lemma}
Let us now prove Lemmas \ref{l:algebraic_topology} and \ref{l:loops_and_duality}.
\begin{proof}[Proof of Lemma \ref{l:algebraic_topology}]
We first argue that we can restrict ourselves to the case where $d=2$. Let $i,j\in\{1,\dots,d\}$ be two distinct coordinates of $H_1(\T^d)\simeq\Z^d$. Let $\calT\subset\T^d$ be a smooth embedding of the two-torus in $\T^d$ whose image in $H_1(\T^d)$ generates the plane corresponding to the coordinates $i$ and $j$. To prove the lemma, it is enough to find a loop $\gamma:S^1\rightarrow\calT\setminus f^{-1}(0)$ with non-trivial homology. On the one hand, this condition is $C^1$-open in $f$. On the other hand, there exist arbitrarily small $C^\infty$ perturbations of $f$ whose zero set is transversal to $\calT$. But such a perturbation will have $0$ as a regular value when restricted to $\calT$. Thus, we can restrict ourselves to the case $d=2$.\\

We now assume that $d=2$ and prove the second point, since it implies the first. We distinguish two cases. First, we assume that there exists $\gamma_0: S^1\rightarrow f^{-1}(0)$ a non-contractible loop. Since $0$ is a regular value of $f$, we can push $\gamma_0$ into $A$ or $B$ by the gradient flow and obtain non-contractible loops in $\textup{Im}(\iota^+_*)$ and $\textup{Im}(\iota^-_*)$ respectively. But since $\pi_1(\T^2)\simeq H_1(\T^2)\simeq\Z^2$, any non-contractible loop defines a non-trivial homology class so both images are non-zero. On the other hand, since $A\cap B=\emptyset$, they are orthogonal for the intersection form. Thus, $\textup{Im}(\iota^+_*)$ and $\textup{Im}(\iota^-_*)$ belong to a common line in $H_1(\T^2)$.\\

Assume now that $f^{-1}(0)$ has only loops that are contractible in $\T^2$ and let us prove that either $\iota^+_*$ is surjective and $\iota^-_*$ vanishes or vice versa. We proceed by induction on the number of loops. If there are no loops then either $A=\T^2$ or $B=\T^2$ so the statement is true. Assume the lemma is true for functions with $n$ loops in their zero set, all of which are contractible, and assume that $f$ has $n+1$ loops in its zero set (all contractible). Since $n+1\geq 1$ at least one such loop exists and bounds a disk $D\subset\T^2$ containig no other loops. Without loss of generality, we may assume that $f$ is negative on $D$. Let $\chi$ be a smooth function that is positive inside $D$ and whose support intersects $f^{-1}(0)$ only in $\partial D$. Then, for $M<+\infty$ large enough, $\tilde{f}=f+ M\chi$ is positive on a neighborhood of $D$ and $\tilde{f}^{-1}(0)=f^{-1}(0)\setminus\partial D$. By induction, for any homology class $\sigma\in H_1(\T^2)$, we may find $\gamma_1:S^1\rightarrow\T^2\setminus\tilde{f}^{-1}(0)$ a smooth loop with homology class $\sigma$, which, by a small $C^\infty$ perturbation, we can assume intersects $\partial D$ transversally. Since $D$ is a disk, $\gamma_1$ is isotopic to a smooth loop $\gamma:S^1\rightarrow \T^2\setminus\br{\tilde{f}^{-1}(0)\cup D}=\T^2\setminus f^{-1}(0)$. But $\gamma_1$ and $\gamma$ have the same homology class $\sigma$. By the intermediate value theorem, $f$ must have constant sign on $\gamma$, which is the same sign as $\tilde{f}$. In particular, the induction hypothesis implies that this sign does not depend on the choice of $\gamma$. Hence, either $\iota^+_*$ is surjective or $\iota^-_*$ is, and, as in the previous case, since their images are orthogonal for the pairing induced by the intersection form, if one is surjective, the other must vanish.
\end{proof}

\begin{proof}[Proof of Lemma \ref{l:loops_and_duality}]
Assume first that $q\in C^\infty(\T^d)$ (so that $f$ is a.s. $C^\infty$) and that for each $x\in\T^d$, $(f(x),\nabla_xf)$ is a non-degenerate Gaussian vector. For $j=1,2$, let $\calL^{j}$ be the set of $u\in C^0(\T^d)$ such that there exists a smooth loop $\gamma: S^1\rightarrow u^{-1}([0,+\infty[)$ whose homology class in $(n_1,\dots,n_d)\in H_1(\T^d)\simeq\Z^d$ satisfies $n_1>0$ if $j=1$ and $n_2>0$ if $j=2$. In particular, $\calL=\calL^1$. Since for each $x\in\T^d$ $(f(x),\nabla_xf)$ is non-degenerate and since the field $(f,\nabla f)$ is a.s. $C^2$, by Bulinskaya's lemma (see Lemma 11.2.10 of \cite{adler_taylor}), a.s., $f$ has no critical points at height $0$. Clearly, for all $j\in\{1,2\}$, $\calL^j$ is admissible (as in Definition \ref{d:admissible_set}) so the threshold maps $\textup{T}_{\calL^j}$ are well defined. By the first part of Lemma \ref{l:algebraic_topology}, we have a.s.,
\[
\min_{j=1,2}\{\textup{T}_{\calL^j}(f)\wedge\textup{T}_{\calL^j}(-f)\}\leq 0\, .
\]
Since $f$ is centered and symmetric, $\textup{T}_{\calL^j}(f)$ and $\textup{T}_{\calL^j}(-f)$ for $j=1,2$ all have the same law and satisfy
\[
\prob\left[T_\calL(f)\leq 0\right]\geq 1/4\, .
\]
This proves the first part of the lemma. To prove the second part, notice that by the second part of Lemma \ref{l:algebraic_topology},
\[
\max_{j=1,2}\{\textup{T}_{\calL^j}(f)\vee\textup{T}_{\calL^j}(-f)\}\geq 0\, .
\]
Reasoning as before, we get
\[
\prob\left[T_\calL(f)\geq 0\right]\geq 1/4
\]
as announced.
\end{proof}
\subsection{Percolation estimates}\label{ss:percolation_estimates}

The object of this subsection will be to establish Lemmas \ref{l:loop_to_cross}, \ref {l:rsw} and \ref{l:gluing}, which we will use in the next subsection. We will often consider fields $f$ defined on $\R^2$ or $\calT_R$. For each $\ell\in\R$, we will denote by $\prob_\ell$ the probability law of $f_\ell:=f+\ell$. We will only specify which field $f$ the notation $\prob_\ell$ is referring to whenever there is a possible ambiguity.

\paragraph{Positive correlation for continuous crossing events}:\\

The Fortuin-Kasteleyn-Ginibre inequality from Bernoulli percolation (see for instance Theorem 2.4 of \cite{Grimmett}) also holds for increasing percolation events of Gaussian fields with suitable regularity and non-degeneracy assumptions, as long as the crucial $q*q\geq 0$ condition holds. This result is essentially due to Pitt (see \cite{pitt_82}) as explained in \cite{rv_17a}. It will be very useful in the proofs of the results of the rest of this section.
\begin{lemma}[FKG inequality for continuous crossings, see Theorem A.4 of \cite{rv_17a}]\label{l:FKG}
Let $f$ be a stationary Gaussian field on $\R^2$ (or $\T^2$) of the form $f=q*W$ where $q\in L^2$ satisfies (Weak) Conditions \ref{cond:regularity} and \ref{cond:positivity} as well as (Strong) Condition \ref{cond:non-degeneracy} and where $W$ is the $L^2$ white noise. Let $A$ and $B$ be two events obtained as unions and intersections of translations the events $\textup{Loop}_R$, $\textup{Cross}_R$, $\textup{Cross}_R^\dagger$ and $\textup{Circ}_x(r_1,r_2)$ defined in Subsection \ref{ss:percolation_estimates} below. Assume that $A$ and $B$ are increasing events\footnote{Here for $A$ to be increasing means that $f\in A\Rightarrow f+\ell\in A$ for any $\ell\geq 0$.}. Then,
\[
\P\brb{A\cap B}\geq \P[A]\P[B]\, .
\]
\end{lemma}
\begin{proof}
This result is presented in \cite{rv_17a} (see Theorem A.4 therein) under some slightly stronger regularity conditions. These conditions come from Lemma A.9 of \cite{rv_17a} but the proof only uses the conditions presented here.
\end{proof}

\paragraph{From loops to widthwise crossings}:\\

Let $\calR_R^\dagger=[0,3R]\times[0,4R]$, which we can see as a subset of $\calT_R$ or $\R^2$. In Lemma \ref{l:loop_to_cross}, we compare the following two percolation events:
\begin{itemize}
\item Let $\textup{Loop}_R$ be the event that there exists $\gamma:S^1\rightarrow f_\ell^{-1}(]0,+\infty[)$ such that the homology class $(n_1,n_2)$ of $\gamma$ in $H_1(\calT_R)\simeq\Z^2$ satisfies $n_1\neq 0$.
\item Let $\textup{Cross}^\dagger_R$ be the event that $f_\ell^{-1}(]0,+\infty[)$ contains a crossing of $\calR_R^\dagger$. We will call this event a \textbf{widthwise crossing}.
\end{itemize}
\begin{lemma}\label{l:loop_to_cross}
Let $f=q*W$  where $q\in L^2(\calT_R)$ satisfies (Weak) Condition \ref{cond:regularity}, (Strong) Condition \ref{cond:non-degeneracy} and Condition \ref{cond:symmetry} and $W$ is the $L^2$ white noise on $\calT_R$. Then, there exists a universal constant $N\in\N$, $N\geq 1$ such that the following holds. For each $p\in[0,1]$, let $\phi_N(p)=1-(1-p)^\frac{1}{N}$ (so that in particular, $\lim_{p\rightarrow 1}\phi_N(p)=1$). Then, for all $R>0$ and $\ell\in\R$,
\begin{equation}\label{e:loop_to_cross}
\prob_\ell\brb{\textup{Cross}^\dagger_R}\geq \phi_N\br{\prob_\ell\brb{\textup{Loop}_R}}\, .
\end{equation}
\end{lemma}
\begin{proof}
Consider $(\calR_j)_{1\leq j\leq 20000}$ the collection of images of $[0,3R]\times[0,4R]$ under the action of translations by vectors of $R\Z^2$ and $\frac{\pi}{2}$ rotations, in $\calT_R$. Since for each $x\in\calT_R$, $(f(x),\nabla_xf)$ is non-degenerate, by Bulinskaya's lemma (see Lemma 11.2.10 of \cite{adler_taylor}), $f$ has a.s. no critical point at height zero. By duality (i.e., by Lemma \ref{l:algebraic_topology}), on the event $\textup{Loop}_R$, $f_\ell^{-1}(]0,+\infty[)$ must (with probability one) contain a widthwise crossing of one of the rectangles $\calR_j$. We call these events $C_j$:
\[
\textup{Loop}_R\Rightarrow\cup_{j=1}^{20000} C_j\, .
\]
Since the events $C_j$ are increasing crossing events, by Lemma \ref{l:FKG}, we get
\[
\P_\ell\brb{\neg\textup{Loop}_R}\geq \prod_{j=1}^{20000}\P_\ell[\neg C_j]=\br{1-\P_\ell\brb{\textup{Cross}_R^\dagger}}^{20000}\, .
\]
In particular, \eqref{e:loop_to_cross} holds for $N=20000$.
\end{proof}

\paragraph{From widthwise crossings to lengthwise crossings}:\\
Let $\calR_R=[0,6R]\times [0,4R]\subset\R^2$. We introduce the two following families of events:
\begin{itemize}
\item For each $0<r_1\leq r_2<+\infty$ and $x\in\R^2$, let $\textup{Ann}_x(r_1,r_2)=\{y\in\R^2\, :\, r_1\leq |x-y|\leq r_2\}$ and let $\textup{Circ}_x(r_1,r_2)$ be the event that there exists a continuous map $\gamma:S^1\rightarrow \textup{Ann}_x(r_1,r_2)\cap f_\ell^{-1}\br{]0,+\infty[}$ whose image separates the two connected components of $\R^2\setminus\textup{Ann}_x(r_1,r_2)$. We call such events \textbf{circuits} and, for brevity, we write $\textup{Circ}(r_1,r_2)=\textup{Circ}_0(r_1,r_2)$.
\item Let $\textup{Cross}_R$ be the event that there exists a continuous map $\gamma:[0,1]\rightarrow\calR_R\cap f_\ell^{-1}(]0,+\infty[)$ such that $\gamma(0)$ belongs to the left side of $\calR_R$ and $\gamma(1)$ belongs to its right side. We will call this event a \textbf{lengthwise crossing}.
\end{itemize}
We will use the following result from \cite{mv_18}, which is an application of Russo-Seymour-Welsh theory to Gaussian fields.
\begin{lemma}[Arm decay, see Theorem 4.7 of \cite{mv_18}]\label{l:rsw}
Let $f=q*W$  where $q\in L^2(\T^d)$ satisfies (Weak) Conditions \ref{cond:regularity}, \ref{cond:non-degeneracy} and \ref{cond:positivity}, as well as Conditions \ref{cond:symmetry} and \ref{cond:decay}, and where $W$ is the $L^2$ white noise on $\R^2$. Then, for all $\ell\geq 0$
\[
\lim_{L\rightarrow+\infty}\inf_{r\geq 1}\prob_\ell\brb{\textup{Circ}(r,L r)}=1\, .
\]
\end{lemma}
This result is actually a variation of Theorem 1.4 of \cite{bg_16}, which was also studied in \cite{bm_18} and \cite{rv_17a}. All of these results are adaptations of the Russo-Seymour-Welsh theory in planar Bernoulli percolation (see for instance Theorem 10.89 of \cite{Grimmett}). In the following lemma, which is inspired by Section 4 of \cite{att_2018}, we combine circuit events and widthwise crossings to produce lengthwise crossings.
\begin{lemma}\label{l:gluing}
Let $f=q*W$  where $q\in L^2(\R^2)$ satisfies (Weak) Conditions \ref{cond:regularity} and \ref{cond:decay}, (Strong) Condition \ref{cond:non-degeneracy} and Condition \ref{cond:symmetry} and $W$ is the $L^2$ white noise. Fix $L\in\N$, $L\geq 1$ and for each $p\in[0,1]$, let $\psi_L(p)=1-(1-p)^{\frac{1}{16L^2}}$ (so that in particular $\lim_{p\rightarrow 1}\psi_L(p)=1$). Then, for each $R>0$ and $\ell\in\R$,
\[
\prob_\ell\brb{\textup{Cross}_{LR}}\geq \psi_L\br{\prob_\ell\brb{\textup{Cross}^\dagger_R}}^2\prob_\ell\brb{\textup{Circ}(R,LR)}\, .
\]
\end{lemma}
\begin{proof}
Throughout the proof, we fix $R>0$ and $L\in\N$, $L\geq 1$. Observe that $\calR_{LR}$ is the union of two copies of $\calR_{LR}^\dagger$ with a common side. We will build a crossing of $\calR_{LR}$ from particular crossings of these two copies of $\calR_{LR}^\dagger$ using a circuit to glue them together. In order to do so, we must first define these particular crossings to have high enough probability. For each $j\in\{0,4L-1\}$, let $I_j=[jR,(j+1)R]$. For each $j_1,j_2\in\{0,4L-1\}$ and each $x=(a,b)\in\R^2$, let $\textup{Cross}^\dagger_x(j_1,j_2)$ be the event that there exists a continuous path $\gamma:[0,1]\rightarrow\br{x+\calR^\dagger_{LR}}\cap f_\ell^{-1}\br{]0,+\infty[}$ such that $\gamma(0)\in\{a\}\times \br{b+I_{j_1}}$ and $\gamma(1)\in\{a+3LR\}\times \br{b+I_{j_2}}$. Then, $\textup{Cross}^\dagger_{LR}=\cup_{j_1,j_2\in\{0,4L-1\}}\textup{Cross}_0^\dagger(j_1,j_2)$ (see Figure \ref{fig:special_cross}). Since $q*q\geq 0$, by the FKG inequality (Lemma \ref{l:FKG}), we deduce that:
\[
\prod_{j_1,j_2\in\{0,4L-1\}}\br{1-\prob_\ell\brb{\textup{Cross}_0^\dagger(j_1,j_2)}}\leq 1-\prob_\ell\brb{\textup{Cross}^\dagger_{LR}}\, .
\]
\begin{figure}\label{fig:special_cross}
\begin{center}
\includegraphics[width=0.3\textwidth]{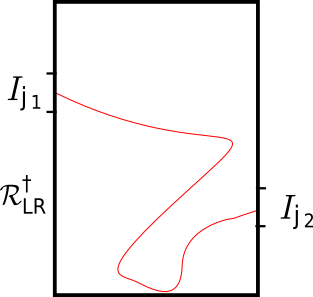}
\end{center}
\caption{An illustration of the event $\textup{Cross}_x^\dagger(j_1,j_2)$.}
\end{figure}
In particular, by stationarity and the symmetry assumption on $q$, there exist $j_1,j_2\in\{0,4L-1\}$, such that for each $x\in\R^2$,
\begin{equation}\label{e:gluing_lemma_1}
\prob_\ell\brb{\textup{Cross}_x^\dagger(j_1,j_2)}=\prob_\ell\brb{\textup{Cross}_x^\dagger(j_2,j_1)}\geq \psi_L\br{\prob_\ell\brb{\textup{Cross}^\dagger_{LR}}}\, .
\end{equation}
Having chosen these indices $j_1$ and $j_2$, we now turn to the gluing construction. Let $x_1=(3LR,0)$ and $x_2=(3LR,j_2R)$. Observe that a path that connects $\{3LR\}\times I_{j_2}$ to the left-hand side of $\calR$ inside $\calR$ or to the right-hand side of $x_1+\calR$ inside $x_1+\calR$ must intersect any loop separating the two connected components of $\R^2\setminus\textup{Ann}_{x_2}(R,LR)$. In particular (see Figure \ref{fig:special_gluing}),
\begin{equation}\label{e:gluing_lemma_2}\textup{Cross}^\dagger_0(j_1,j_2)\cap\textup{Cross}^\dagger_{x_1}(j_2,j_1)\cap \textup{Circ}_{x_2}(R,LR)\subset \textup{Cross}_{LR}\, .
\end{equation}
\begin{figure}\label{fig:special_gluing}
\begin{center}
\includegraphics[width=0.3\textwidth]{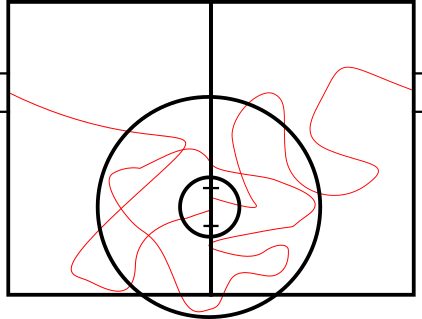}
\end{center}
\caption{Two well chosen widthwise crossings and a circuit form a lengthwise crossing.}
\end{figure}
By symmetry assumption on $q$ and \eqref{e:gluing_lemma_1}, the probability of the left-hand side of \eqref{e:gluing_lemma_2} is at least
\[
\psi_L\br{\prob_\ell\brb{\textup{Cross}^\dagger_{LR}}}^2\prob_\ell\brb{\textup{Circ}(R,LR)}
\]
which concludes the proof.
\end{proof}

\subsection{The conclusion}\label{ss:the_conclusion}

In this subsection we prove Proposition \ref{p:phase_transition} by relying on the results established in Subsections \ref{ss:loop_corollary} and \ref{ss:percolation_estimates}.\\

We will use the rectangles $\calR_R^\dagger=[0,3R]\times[0,4R]$ and $\calR_R=[0,6R]\times [0,4R]$ as well as the events $\textup{Cross}^\dagger_R$, $\textup{Cross}_R$, $\textup{Loop}_R$ and $\textup{Circ}(r_1,r_2)$ introduced in the previous subsection. More precisely, we will consider events defined in exactly the same way but with $f_\ell$ replaced by another field. We choose to keep the same notation for clarity and specify which field we are considering whenever any ambiguity is possible.

\begin{proof}[Proof of Proposition \ref{p:phase_transition}]
Throughtout the proof, we fix $\ell>0$.\\

\textit{Step 1: defining approximations of the field}\\
In this step we approximate the field by a periodic field. To this end, we introduce $\chi\in C^\infty(\R^2)$ equal to one on $[-6,6]^2$ and compactly supported in $[-12,12]^2$ and letting $q_R=q\chi(\cdot/R)$. Next we let $\tilde{q}_R:\calT_R\rightarrow\R$ be defined by $\tilde{q}_R(y)=\sum_{x\textup{ mod }100R\Z^2=y} q_R(x)$ (here the sum is actually finite because $\chi$ is compactly supported). We then introduce $W$ and $\widetilde{W}$ the white noise on $\R^2$ and $\calT_R$ respectively. Finally, we let $f_\ell=q*W+\ell$, $f_{\ell,R}=q_R*W+\ell$ and $\tilde{f}_{\ell,R}=q_R*\widetilde{W}+\ell$, whose laws we denote respectively by $\prob_\ell$, $\prob_{\ell,R}$ and $\widetilde{\prob}_{\ell,R}$ respectively. We make the three following observations. First, the definition of $q_R$ implies that $f_{\ell,R}$ and $\tilde{f}_{\ell,R}$ have the same law on the rectangle $\calR_R$. Second, by Remarks \ref{rk:sup_1} and \ref{rk:sup_2},
\begin{equation}\label{e:sharp_1}
\lim_{R\rightarrow+\infty}\E\brb{\sup_{x\in\calR_R}|f_{\ell,R}(x)-f_\ell(x)|}=0\, .
\end{equation}
Indeed, Condition \ref{cond:decay} implies that $\|q-q_R\|_{W^{1,2}}$ decay polynomially in $R$ so, by Remark \ref{rk:sup_1}, $\|(q-q_R)*(q-q_R)\|_{C^{1;1}}$ also decays polynomially in $R$. Therefore, applying Remark \ref{rk:sup_2} to the field $(q-q_R)*W$ on the ball of radius $1000R$ centered at $0$ yields \eqref{e:sharp_1}. Third, the vector $(f(0),\nabla_0f)$ is non-degenerate, and so are the vectors $(f_{\ell,R}(0),\nabla_0f_{\ell,R})$ and $(\tilde{f}_{\ell,R}(0),\nabla_0\tilde{f}_{\ell,R})$ for all large enough values of $R$. For the first vector this follows from Remark \ref{rk:non-dg_is_free}. Since the non-degeneracy condition is open, the condition must also be true for the other vectors by continuity.

\textit{Step 2: proving widthwise crossings are likely}\\
We wish to apply Corollary \ref{c:loop_percolation} to $\tilde{f}_{0,R}$ with $\eps=\ell/3$. In particular, we will use the set $\calL$ and the functional $\alpha(\cdot)$ introduced in the statement of the corollary. Since $q\in L^1(\R^2)$ is non-zero, neither is $q_R$ for large enough values of $R$. Moreover, notice that $\tilde{f}_{0,R}+\ell/3\in\calL$ holds if and only if $\textup{Loop}_R$ holds for $\tilde{f}_{\ell/3,R}$. Finally, if we define $\alpha(\tilde{q}_R)$ as in \eqref{e:alpha}, we have $\alpha(\tilde{q}_R)\xrightarrow[R\to+\infty]{}+\infty$. Indeed, by Condition \ref{cond:decay}, for any $p\in\{1,2\}$ and for some $\beta>2$ and $C<+\infty$,
\begin{align*}
\int_{\calT_R}\left|\sum_{v\in\Z^2\setminus\{0\}}q(x+100R v)\chi(x/R)\right|^pdx&\leq C\|\chi\|_\infty^p\int_{|x|\leq 24 R}\left(\sum_{v\in\Z^2\setminus\{0\}}|x+100R v|^{-\beta}\right)^pdx\\
&\leq 1/50)2^{\beta p}(24)^2C\|\chi\|_\infty^p R^{2-\beta p}\sum_{v\in\Z^2\setminus\{0\}}|v|^{-\beta}\\
&=O(R^{2-\beta p})
\end{align*}
so that $\|\tilde{q}_R\|_{L^p(\calT_R)}\xrightarrow[R\to+\infty]{}\|q\|_{L^p(\R^2)}$. Since $\sqrt{\calT_R}\xrightarrow[R\to+\infty]{}+\infty$, we indeed have $\lim_{R\rightarrow+\infty}\alpha(\tilde{q}_R)=+\infty$ as $R\rightarrow+\infty$. Thus, Corollary \ref{c:loop_percolation} shows that
\[
\lim_{R\rightarrow+\infty}\widetilde{\prob}_{\ell/2,R}\brb{\textup{Loop}_R}=1\, .
\]
By Lemma \ref{l:loop_to_cross}, which applies because of the third observation made in Step 1 of the present proof, and thanks to the first observation, we therefore have $\lim_{R\rightarrow+\infty}\prob_{\ell/2,R}\brb{\textup{Cross}^\dagger_R}=1$. Finally, combining this estimate with \eqref{e:sharp_1} yields:
\begin{equation}\label{e:sharp_2}
\lim_{R\rightarrow+\infty}\prob_\ell\brb{\textup{Cross}^\dagger_R}=1\, .
\end{equation}

\textit{Step 3: from widthwise crossings to lengthwise crossings}\\
Let $\eps>0$. By Lemma \ref{l:rsw}, there exists $L\in\N$, $L\geq 1$ such that for each $r\geq 1$,
\begin{equation}\label{e:sharp_3}
\prob_\ell\brb{\textup{Circ}(r,Lr)}\geq (1-\eps)^{\frac{1}{3}}\, .
\end{equation}
Now, let $\psi_L$ be as in Lemma \ref{l:rsw}. By \eqref{e:sharp_2}, there exists $r_0\geq 1$ such that for all $r\geq r_0$,
\begin{equation}\label{e:sharp_4}
\psi_L\br{\prob_\ell\brb{\textup{Cross}^\dagger_r}}\geq (1-\eps)^{\frac{1}{3}}\, .
\end{equation}
Set $R_0=Lr_0$. Then, for each $R\geq R_0$, $r:=R/L\geq r_0\geq 1$ so by Lemma \ref{l:gluing} (which applies by the third observation of Step 1), \eqref{e:sharp_3} and \eqref{e:sharp_4} we have
\[
\prob_\ell\brb{\textup{Cross}_R}\geq \psi_L\br{\prob_\ell\brb{\textup{Cross}^\dagger_r}}^2\prob_\ell\brb{\textup{Circ}(r,Lr)}\geq 1-\eps\, .
\]
Since this is true for all $\eps>0$, the proof is over.
\end{proof}

\appendix

\section{Standard results on Gaussian fields}\label{s:appendix}

The results presented in this appendix are well known, though we were unable to find a reference presenting them in the setting of the present paper.

\subsection{Orthogonal expansions and the Cameron-Martin space}

Let $\gamma$ be a Gaussian measure on $\calT$ and let $f$ be a Gaussian field on $\calT$ with law $\gamma$, defined on some probability space $(\Omega,\calF,\prob)$ and let $G\subset\Omega$ be the $L^2$-closure of the set of variables $f(x)$ for $x\in\calT$. Then, the \textbf{Cameron-Martin space} of $f$ is the space $H_\gamma$ of functions $h_\xi:x\mapsto\E[\xi f(x)]$ for $\xi\in G$. If $H_\gamma$ is finite dimensional, then $f$ has the law of a standard Gaussian vector in $H_\gamma$ in the following sense. Let $(e_j)_{j\in\{1,\dots,n\}}$ be an orthonormal basis in $H_\gamma$. First, $f\in H$ a.s. Moreover, the family $(\langle f,e_j\rangle)_{j\in\{1,\dots,n\}}$ is a family of independent standard normals.
\begin{lemma}\label{l:CM_basis}
Let $\xi_1,\dots,\xi_n$ be indepenent standard normals and let $e_1,\dots,e_n\in C^0(\calT)$ be such that $(e_k)_{k\in\{1,\dots,n\}}$ is linearly independent. Let $f=\sum_{j=1}^n \xi_j e_j$. Then, the family $(e_k)_{k\in\{1,\dots,n\}}$ is an orthonormal basis for the Cameron-Martin space of $f$.
\end{lemma}
\begin{proof}
Let $H$ be the Cameron-Martin space of $f$. Let $k\in\{1,\dots,n\}$. Since $(e_j)_{j\in\{1,\dots,n\}}$ is linearly independent, there exists $x_k\in\calT$ such that for each $j\in\{1,\dots,n\}$, $e_j(x)=\delta_{jk}$. In particular, $\E[f(x_k)f]=e_k$ belongs to $H$. Conversly, since for each $x\in\calT$, $\E[f(x)f]=\sum_{k=1}^n e_k(x)e_k$, $(e_j)_{j\in\{1,\dots,n\}}$ generates $H$. Now, let $j,k\in\{1,\dots,n\}$. Then, $\langle e_j,e_k\rangle=\E[f(x_j)f(x_k)]=e_j(x_k)=\delta_{jk}$ so that $(e_l)_{l\in\{1,\dots,n\}}$ is an orthonormal family in $H$.
\end{proof}

\subsection{Supremum bounds}

We $\calT$ equip with the distance associated to the metric inherited from $\R^d$ which we denote by $\textup{dist}$. For each $l\in\N$ and $\nu\in]0,1[$ let $C^{l,\nu}(\calT)$ be the space of $l$ times differentiable functions $q:\calT\rightarrow\R$ whose derivatives of order $l$ are of H\"older class $\nu$. We equip this space with the usual norm:
\[
\|q\|_{C^{l,\nu}(\calT)}=\max_{|\alpha|\leq l}\brb{\max_{x\in\calT}|\partial^\alpha q(x)|+\sup_{x\neq y}\textup{dist}(x,y)^{-\nu}|\partial^\alpha q(x)-\partial^\alpha q(y)|}\, .
\]
Moreover, we let $C^{l,\nu;l,\nu}(\calT)$ be the space of functions $K:\calT\times\calT\rightarrow\R$ such that for each $\alpha,\beta\in\N^d$ with $|\alpha|,|\beta|\leq l$, $\partial^{\alpha,\beta}K$ exists and is of H\"older class $C^\nu$ in the joint variables. We equip this space with the norm $\|K\|_{C^{l,\nu;l,\nu}(\calT)}$ defined as follows:
\[
\max_{|\alpha|,|\beta|\leq l}\brb{\sup_{x,y\in\calT} |\partial^{\alpha,\beta}K(x,y)|+\sup_{(x,y)\neq(x',y')}|\partial^{\alpha,\alpha}K(x,y)-\partial^{\alpha,\alpha}K(x',y')|\br{\textup{dist}(x,y)+\textup{dist}(x',y')}^{-\nu}}\, .
\]
Here and below, $\partial^{\alpha,\beta}$ means the partial derivative $\partial^\alpha$ is applied to the first variable while $\partial^\beta$ is applied to the second variable. The following lemma is a well known fact on the sample path regularity of Gaussian fields.
\begin{lemma}[see Corollary 1.7 of \cite{azais_wschebor} together with Appendix A of \cite{ns_16}]\label{l:basic_regularity}
Let $l\in\N$ and $\nu\in]0,1[$. Let $f$ be a Gaussian field on $\calT$ with covariance function $K\in C^{l,\nu;l,\nu}(\calT)$. Then, $f$ is almost surely of class $C^l$. Moreover, for each $\alpha,\beta\in\N^d$ such that $|\alpha|,|\beta|\leq l$ and each $x,y\in\calT$
\[
\E\brb{\partial^\alpha f(x)\partial^\beta f(y)}=\partial^{\alpha,\beta}K(x,y)\, .
\]
\end{lemma}
Next we need a lemma to control the covariance of a stationary field in terms of its covariance square root.
\begin{lemma}\label{l:sup_convolution}
Let $l\in\N$ and $\nu\in]0,1[$, let  $q\in C^{l,\nu}(\calT)$. Then, the convolution $\kappa(x)=q*q(x)=\int_\calT q(y) q(x-y)dy$ exists there exists $C=C(l,d,\nu)<+\infty$ such that
\[
\|\kappa\|_{C^l(\calT)}\leq C\|q\|_{W^{l,2}}^2\textup{ and }\|\kappa\|_{C^{l,\nu}(\calT)}\leq C \|q\|_{C^{l,\nu}(\calT)}\|q\|_{W^{l,1}(\calT)}\, .
\]
\end{lemma}
The first upper bound is useful because it requires a slower decay of $q$ to work while the second one is useful because it captures the H\"older regularity of the covariance.
\begin{proof}
First of all, for all $\alpha,\beta\in\N^d$ with $|\alpha|,|\beta|\leq l$ and each $x\in\calT$, by integration by parts,
\[
|\partial^{\alpha+\beta}\kappa(x)|\leq\int_\calT|\partial^\alpha q(y)\partial^\beta q(x-y)|dy\leq \|q\|_{W^{l,2}}^2\leq \|q\|_{W^{l,1}}\|q\|_{C^l}\, .
\]
Observe that the first inequality shows the first statement of the lemma. Next, for each $x,y\in\calX$ distinct,
\[
|\partial^{\alpha+\beta}\kappa(x)-\partial^{\alpha+\beta}\kappa(y)|\leq\int_\calT|\partial^{\alpha+\beta} q(z)||q(x-z)-q(y-z)|dz\leq \|q\|_{W^{l,1}}\|q\|_{C^{l,\nu}}\textup{dist}(x,y)^{\nu}
\]
so that $\|\partial^{\alpha+\beta}\kappa\|_{C^{\nu}}\leq \|q\|_{W^{l,1}}\|q\|_{C^{l,\nu}}$. Summing the two inequalities over the possible values of $\alpha+\beta$ with $|\alpha+\beta|\leq l$ we get the result.
\end{proof}
\begin{remark}\label{rk:sup_1}
The result of Lemma \ref{l:sup_convolution} holds, with exactly the same proof, if we replace $\calT$ by $\R^2$.
\end{remark}
In this third lemma, we show that we can control the supremum of the field (and its derivatives) in terms of the H\"older norms of its covariance. In particular, only H\"older regularity is needed to obtain upper bounds for the decay of the field.
\begin{lemma}\label{l:sup_lemma}
Assume that in the expression $\calT=\R/(2\pi R_1\Z)\times\dots\R/(2\pi R_d\Z)$ defining $\calT$, for all $i\in\{1,\dots,d\}$, $R_i\geq 1$. Let $l\in\N$ and $\nu\in]0,1[$. Let $f$ be a Gaussian field on $\calT$ with covariance function $K\in C^{l,\nu;l,\nu}(\calT)$.
Let $M_{K,l,\nu}=\|K\|_{C^{l,\nu;l,\nu}(\calT)}$. Then, there exists $C=C(d,l,\nu)<+\infty$ such that
\begin{equation}\label{e:expected_supremum_bound}
\E\brb{\|f\|_{C^l}}\leq C M_{K,
l,\nu}^{1/2}\brb{1+|\ln\br{|\calT|}|^{1/2}+|\ln\br{M_{K,l,\nu}}|^{1/2}}\, .
\end{equation}
Moreover, for each $\lambda\geq 0$,
\begin{equation}\label{e:btis}
\prob\brb{\|f\|_{C^l(\calT)}\geq \E\brb{\|f\|_{C^l(\calT)}}+\lambda}\leq 2 e^{-\frac{\lambda^2}{2|\calT|M_{K,l,\nu}^2}}\, .
\end{equation}
\end{lemma}
\begin{remark}\label{rk:sup_2}
The result of Lemma \ref{l:sup_lemma} holds, with exactly the same proof, if we replace $\calT$ with a Euclidean ball of radius at least one in $\R^d$.
\end{remark}
The expectation bound will follow from classical results for Gaussian processes (Theorem 11.18 of \cite{ledoux_talagrand}) and the tail bound will then follow from the Borell-TIS inequality (see Theorem 2.1.1 of \cite{adler_taylor}).
\begin{proof}
We start by proving the expectation bound. To bound $\E\brb{\|f\|_{C^l(\calT)}}$ it is enough to bound the expected supremum of each partial derivative $\partial^\alpha f$ for $|\alpha|\leq l$. Since all of the derivatives of $K$ up to $l$ in each variable are $\nu$-H\"older, this reduces the problem to the case $l=0$. If $M_{K,0,\nu}=0$ then $f$ is a.s. constant and so its supremum is centered so the bound is trivially satisfied so we assume $M_{K,0,\nu}>0$. In this case, we are ready to apply Theorem 11.18 of \cite{ledoux_talagrand}. To this end, let $d_K$ be the canonical pseudo-metric on $\calT$ associated to the field $f$ restricted to $\calT$: for each $x,y\in\calT$, $d_K(x,y):=\E\brb{(f(x)-f(y))^2}^{1/2}=\br{K(x,x)+K(y,y)-2K(x,y)}^{1/2}$ and for each $x\in\calT$ and $r>0$, let $B_K(x,r)$ be the set of $y\in\calT$ such that $d_K(x,y)\leq r$. It should be distinguished from $B(x,r)$ be the metric ball of $\calT$ of radius $r$ centered at $x$. By Theorem 11.18 of \cite{ledoux_talagrand}, since $f$ is stationary, there is an absolute constant $C_1<+\infty$ such that
\[
\E\brb{\sup_\calT |f|}\leq C_1\sup_{x\in\calT}\int_0^{+\infty}\brb{\ln\br{\frac{|\calT|}{|B_K(x,r)|}}}^{1/2}dr\, .
\]
For each $x,y\in\calT$, by the triangle inequality and the H\"older condition respectively, we have
\[
d_K(x,y)\leq 2M_{K,0,\nu}^{1/2};\ d_K(x,y)\leq\br{2M_{K,0,\nu}}^{1/2}\textup{dist}(x,y)^{\nu/2}\, .
\]
Thus, if $r\geq 2M_{K,0,\nu}^{1/2}$, $B_K(x,r)=\calT$ and for all $r>0$, $B\br{x,\rho(r)} \subset B_K(x,r)$ where $\rho(r)=(2M_{K,0,\nu})^{-1/\nu}r^{2/\nu}$.
These observations imply that the integrand vanishes for $r\geq 2(M_{K,0,\nu})^{1/2}$ so that
\[
\E\brb{\sup_\calT |f|}\leq C_1\int_0^{2M_{K,0,\nu}^{1/2}}\brb{\ln\br{\frac{|\calT|}{|B\br{x,\rho(r)}|}}}^{1/2}dr
\]
where $x$ is any point of $\calT$. Using $\sqrt{a+b}\leq \sqrt{a}+\sqrt{b}$ for $a,b\geq0$ repeatedly, we see that for some $C_2=C_2(\nu)$,
\[
\E\brb{\sup_\calT|f|}\leq C_2 M_{K,0,\nu}^{1/2}\brb{|\ln\br{|\calT|}|^{1/2}+\ln\br{M_{K,0,\nu}}^{1/2}}+\int_0^{2 M_{K,0,\nu}^{1/2}}|\ln(r)|^{1/2}dr\, .
\]
Since $\int_0^x |\ln(s)|^{1/2}ds=O(x(1+|\ln(x)|))$, we get, for some $C_3=C_3(d,l,\nu)<+\infty$,
\[
\E\brb{\sup_\calT |f|}\leq C_3 M_{K,0,\nu}^{1/2}\brb{1+|\ln\br{|\calT|}|^{1/2}+|\ln\br{M_{K,0,\nu}}|^{1/2}}\, .
\]
This is exactly \eqref{e:expected_supremum_bound} for $l=0$. As discussed above the general case follows readily. The bound \eqref{e:btis} follows from Theorem 2.1.1 of \cite{adler_taylor}. Indeed, let $J_l=\{\alpha\in\N^d \st |\alpha|\leq l\}$. We can define a Gaussian field on $\calT\times J_l$ defined as $(x,\alpha)\mapsto \partial^\alpha f(x)$. Then, $\|f\|_{C^l(\calT)}$ is the sup-norm of this process. The maximal pointwise variance of this process is bounded by $M_{K,l,\nu}^2$ so the aforementioned theorem applies and gives \eqref{e:btis}.
\end{proof}

\subsection{An approximation result}

For each $l\in\N$ and $\nu\in]0,1[$, let $\Gamma_l(\calT)$ (resp. $\Gamma_{l,\nu}(\calT)$ be the space of Gaussian measures on $C^l(\calT)$ (resp. $C^{l,\nu}(\calT)$) equipped with the topology of weak-* convergence in $C^l$ (resp. $C^{l,\nu}$) topology. For each $\gamma\in\Gamma_0(\calT)$, let $H_\gamma$ be the Cameron-Martin space of $\gamma$.

\begin{lemma}\label{l:cv_gamma_to_k}
Fix $l\in\N$ and $\nu\in]0,1[$. Let $(\gamma_n)_{n\in\N}$ be a sequence of Gaussian measures in $\Gamma_{l,\nu}(\calT)$  (resp. $\Gamma_l(\calT)$) converging to $\gamma$ in $\Gamma_{l,\nu}(\calT)$ (resp. $\Gamma_l(\calT)$). Let $K$ be the covariance function of the field defined by $\Gamma$ and for each $n\in\N$, let $K_n$ be the covariance of the field defined by $\gamma_n$. Then, the sequence $(K_n)_{n\in\N}$ converges to $K$ in $C^{l,\nu;l,\nu}(\calT)$ (resp. $ C^{l;l}(\calT)$).
\end{lemma}
\begin{proof}
Let us first assume convergence in $\Gamma_l(\calT)$. By Skorokhod's representation theorem, we can find a sequence $(f_n)_{n\in\N}$ of Gaussian fields converging a.s. to a Gaussian field $f$ in $C^{l,\nu}$-norm such that $f$ has law $\gamma$ and for each $n\in\N$, $f_n$ has law $\gamma_n$. Let $\alpha,\beta\in\N^d$ with $|\alpha|,|\beta|\leq l$. Let $\varphi:\R^2\rightarrow\R$ be continuous function with compact support. Then,
\begin{multline*}
\sup_{x,y\in\calT}\left|\E\brb{\varphi\br{\partial^\alpha f_n(x),\partial^\beta f_n(y)}}-\E\brb{\varphi\br{\partial^\alpha f(x),\partial^\beta f(y)}}\right|\\
\leq \E\brb{\sup_{x,y}\left|\varphi\br{\partial^\alpha f_n(x),\partial^\beta f_n(y)}-\varphi\br{\partial^\alpha f(x),\partial^\beta f(y)}\right|}
\end{multline*}
which tends to $0$ as $n\rightarrow+\infty$ since $\varphi$ is uniformly continuous and $(\partial^\alpha f_n,\partial^\beta f_n)$ converges uniformly on $\calT\times\calT$. Thus, the vectors $(\partial^\alpha f_n(x),\partial^\beta f_n(y))$ converge to $(\partial^\alpha f(x),\partial^\beta f(y))$ uniformly in $(x,y)$. But this implies that $\partial^{\alpha,\beta} K_n$ converges uniformly to $\partial^{\alpha,\beta}K$ as $n\rightarrow+\infty$. Thus, convergence in $\Gamma_l$ implies convergence of covariances in $C^{l;l}$. Assume now that $(\gamma_n)_{n_\in\N}$ converges in $\Gamma_{l,\nu}(\calT)$ to $\gamma$. As before, the supremum over $x, x',y\in\calT$, $x\neq x'$ of the following quantity tends to $0$ as $n\rightarrow+\infty$.
\begin{multline*}
\Big|\E\brb{\varphi\br{\textup{dist}(x,x')^{-\nu}\br{\partial^\alpha f_n(x)-\partial^\alpha f_n(x')},\partial^\beta f_n(y)}}-\\
\E\brb{\varphi\br{\textup{dist}(x,x')^{-\nu}\br{\partial^\alpha f(x)-\partial^\alpha f(x')},\partial^\beta f(y)}}\Big|\, .
\end{multline*}
This shows that the map $(x,x',y)\mapsto\textup{dist}(x,x')^{-\nu}\br{\partial^{\alpha,\beta}K_n(x,y)-\partial^{\alpha,\beta}K_n(x',y)}$ converge uniformly as $n\rightarrow+\infty$. By symmetry we conclude that $\partial^{\alpha,\beta}K_n$ converges to $\partial^{\alpha,\beta}K$ in the $C^{0,\nu;0,\nu}$ topology as $n\rightarrow+\infty$. This concludes the proof of the lemma.
\end{proof}

\bibliographystyle{plain}
\bibliography{sources}

\end{document}